\documentclass[10pt,reqno]{amsart}
\usepackage{epsfig}
\usepackage{color}
\usepackage{esint}
\usepackage[colorlinks=true]{hyperref}
\usepackage{paralist}

\usepackage{mathrsfs}

\allowdisplaybreaks
\numberwithin{equation}{section}

\def\L{{\mathcal L}}
\def\Ln{{\mathcal L}^n}
\def\N{{\bf N}}
\def\R{{\bf R}}
\def\Rn{{\bf R}^n}

\def\Z{{\bf Z}}

\def\={:=}
\def\eps{{\varepsilon}}

\def\FF{{\mathcal F}}

\def\FFepsj{{\mathcal F}_{j}}
\def\Fj{F_j}

\def\dx{\, dx}

\def\dist{{\rm dist}}

\def\Teioe{T_{j}^{\ii}(\om)}

\def\Teoe{T_{j}(\om)}
\def\Teo{T_{j}(\om)}
\def\mmu{\mathbb{P}}
\def\om{\omega}
\def\gammaio{\gamma(\ii,\om)}
\def\gammab{\gamma_0}
\def\epsj{\eps_j}
\def\cnps{c\,}
\def\cnpsm{c\,}
\def\Wsp{W^{s,p}}
\def\Ksp{\dot{W}^{s,p}}
\def\tWsp{\hat{W}^{s,p}}
\def\Wspo{W_0^{s,p}}
\def\csp{\mathrm{cap}_{s,p}}
\def\cspk{\mathrm{cap}_{\KK}}
\def\CSP{C_{s,p}}
\def\CSPK{C_{\KK}}
\def\Om{\Omega}
\def\Omp{O}
\def\deltaj{\delta_j}

\def\Qij{Q_j^{\ii}}
\def\Cij{C_j^{\ii,h_j}}
\def\Bij{B_j^{\ii}}
\def\Bkj{B_j^{\kk}}
\def\tBij{\hat{B}^{\ii}_j}
\def\tBkj{\hat{B}^{\kk}_j}
\def\Kij{T_j^{\ii}}
\def\Kj{T_j}
\def\Bijh{B_j^{\ii,h}}
\def\Bijhj{B_j^{\ii,h_j}}

\def\Cijh{C_j^{\ii,h}}

\def\uij{u^{\ii}_j}
\def\ukj{u^{\kk}_j}

\def\res{\mathop{\hbox{\vrule height 7pt width .4pt depth 0pt
\vrule height .4pt width 6pt depth 0pt}}\nolimits}
\newcommand{\defectsp}{\mathcal{D}_{s,p}}
\def\tildeu{\tilde{u}}

\def\AA{\mathcal A}
\def\BB{\mathcal B}
\def\ii{{\tt i}}
\def\kk{{\tt k}}
\def\lll{{\tt l}}
\def\rL{r_\Lambda}

\def\RL{R_\Lambda}
\def\diam{\mathrm{diam}}
\def\A{A}

\def\Ac{U\setminus\overline{A}}
\def\Apr{A^\prime}
\def\Lip{\mathrm{Lip}}
\def\xii{{\mathrm x}^{\ii}}
\def\xkk{{\mathrm x}^{\kk}}
\def\xiij{\xii_j}
\def\xkkj{\xkk_j}
\def\Vij{V_{j}^{\ii}}
\def\lambdaj{\lambda_j}
\def\xij{\xi_{N}}

\def\KK{\mathcal K}
\def\KKK{\mathscr K}

\def\KKepsj{\KK_{j}}
\def\rj{r_j}
\def\Rj{R_j}
\def\Ieps{{\mathcal I}_j}
\def\IIeps{{\mathscr I}_j}
\def\IL{{\mathcal I}_\Lambda}
\def\IIL{{\mathscr I}_\Lambda}

\def\spt{{\rm spt}}
\def\res{\mathop{\hbox{\vrule height 7pt width .4pt depth 0pt
\vrule height .4pt width 6pt depth 0pt}}\nolimits}

\def \trait (#1) (#2) (#3){\vrule width #1pt height #2pt depth #3pt}
\def \qed{\hfill
        \trait (0.1) (6) (0)
        \trait (6) (0.1) (0)
        \kern-6pt
        \trait (6) (6) (-5.9)
        \trait (0.1) (6) (0)
\medskip}

\newtheorem{theorem}{Theorem}[section]
\newtheorem{lemma}[theorem]{Lemma}
\newtheorem{corollary}[theorem]{Corollary}
\newtheorem{remark}[theorem]{Remark}
\newtheorem{example}[theorem]{Example}
\newtheorem{proposition}[theorem]{Proposition}
\newtheorem{definition}[theorem]{Definition}

\begin{document}
\title[{Aperiodic fractional obstacle problems}]
{Aperiodic fractional obstacle problems}
\author{Matteo Focardi}
\address{Dip. Mat. ``Ulisse Dini''\\
         V.le Morgagni, 67/a\\ 
         50134  Firenze, Italy} 
\email{focardi@math.unifi.it}
\urladdr{http://web.math.unifi.it/users/focardi/}
\keywords{Obstacle problems; Non-local energies; Delone sets of points; 
$\Gamma$-convergence.}
\subjclass[2000]{Primary: 74Q15, 35R11, 49J40,}

\begin{abstract}
We determine the asymptotic behaviour of (bilateral) obstacle problems 
for fractional energies in rather general aperiodic settings via 
$\Gamma$-convergence arguments.
As further developments we consider obstacles with random sizes 
and shapes located on points of standard lattices, and the case 
of random homothetics obstacles centered on random Delone sets 
of points.   

Obstacle problems for non-local energies occur in several 
physical phenomenona, for which our results provide a description 
of the first order asympotitc behaviour.
\end{abstract}
\maketitle

\section{Introduction}\label{intro}

Non-local energies and operators have been actively investigated 
over recent years. They arise in problems from different fields, 
the most celebrated being Signorini's problem in contact mechanics: 
finding the equilibria of an elastic body partially 
laying on a surface and acted upon part of its boundary 
by unilateral shear forces (see \cite{Signorini}, \cite{Fichera}).
In the anti-plane setting the elastic energy can be then expressed 
in terms of the seminorm of a $H^{1/2}$ function, or equivalently 
as the boundary trace energy of a $W^{1,2}$ displacement. 

As further examples we mention applications in elasticity, for 
instance in phase field theories for dislocations (see 
\cite{KCO} and the references therein);
in heat transfer for optimal control of temperature across a surface 
\cite{Frehse}, \cite{Athanasopoulos}; in equilibrium statistical 
mechanics to model free energies of Ising spin systems with Kac 
potentials on lattices (see \cite{ABCP} and the references therein);
in fluid dynamics to describe flows through semi-permeable membranes 
\cite{Duvaut-Lions};
in financial mathematics in pricing models for American options
\cite{Amadori}; and in probability in the theory of Markov processes 
(see \cite{BK1}, \cite{BK2} and the references therein).

Many efforts have been done to extend the existing theories for 
(fully non-linear) second order elliptic equations to non-local 
equations.
Regularity has been developed for integro-differential 
operators \cite{BK1}, \cite{BK2}, \cite{Caf-Silv1}, 
and for obstacle problems for the fractional laplacian 
(see \cite{Caf-Sal-Silv}, \cite{Silvestre} and the references therein).
Periodic homogenization has been studied for a quite general class of 
non-linear, non-local uniformly elliptic equations \cite{Schwab} 
and for obstacle problems for the fractional laplacian 
\cite{Caf-Mel2}, \cite{Foc}. 

In this paper we determine the homogenization limit for bilateral 
obstacle problems involving fractional type energies. In doing that 
we employ a variational approach by following De Giorgi's 
$\Gamma$-convergence theory. 
Many contributions in literature are related to the analogous problem 
for energies defined in ordinary Sobolev spaces or equivalently for
local elliptic operators. 
The setting just mentioned will be referred to in 
the sequel as the \emph{local case} in contrast to the non-local 
framework object of our analysis.
Starting from the seminal papers by Marchenko 
and Khruslov \cite{MK}, Rauch and Taylor \cite{RT1}, \cite{RT2}, and 
Cioranescu and Murat \cite{CM} there has been an outgrowing interest 
on this kind of problems with different approaches (see the books 
\cite{ATT}, \cite{B3}, \cite{BDF}, \cite{CD}, \cite{DM} and the 
references therein on this subject).  
We limit ourselves to stress that $\Gamma$-convergence theory 
was successfully applied to tackle the problem and to solve it  
in great generality (see \cite{DGDML}, \cite{DM2}, \cite{DM3}).

Let us briefly resume the contents of this paper in a model case 
(for all the details and the precise assumptions see section~\ref{s:determ}).
With fixed a bounded set $T$, and a discrete and homogeneous 
distribution of points 
$\Lambda=\{{\mathrm x}^\ii\}_{\ii\in\Z^n}$ (see Definition~\ref{voro}), 
define for all $j\in\N$ the \emph{obstacle set} $\Kj\subseteq\Rn$ by 
$\Kj=\cup_{\ii\in\Z^n}\left(\epsj\,{\mathrm x}^\ii+\epsj^{n/(n-sp)}T\right)$, 
where $(\epsj)_{j\in\N}$ is a positive infinitesimal sequence. 
Given a bounded, open and connected subset $U$ of $\Rn$, $n\geq 2$, 
with Lipschitz regular boundary we consider the functionals 
$\FFepsj:L^p(U)\to[0,+\infty]$ given by
$$
\FFepsj(u):=\int_{U\times U}\frac{|u(x)-u(y)|^p}{|x-y|^{n+sp}}dxdy
\quad \text{ if } u\in \Wsp(U),\,
\tildeu=0\,\,  \csp \text{ q.e. on } \Kj\cap U
$$
and $+\infty$ otherwise. Here, $\Wsp(U)$ is the Sobolev-Slobodeckij 
space for $s\in(0,1)$, $p\in(1,+\infty)$ and $sp\in(1,n)$, 
$\csp$ is the related variational $(p,s)$-capacity, and $\tilde{u}$ 
denotes the precise representative of $u\in\Wsp(U)$ which is defined 
except on a $\csp$-negligible set (see subsections~\ref{sobslob} 
and \ref{cap}).

In Theorem~\ref{main} we show that the asymptotic behaviour of
the sequence $(\FFepsj)_{j\in\N}$ is described in terms of 
$\Gamma(L^p(U))$-convergence by the functional
$\FF:L^p(U)\to[0,+\infty]$ defined by
$$
\FF(u)=\int_{U\times U}\frac{|u(x)-u(y)|^p}{|x-y|^{n+sp}}dxdy
+\theta\,\csp(T)\int_{U}|u(x)|^p\beta(x)\,dx
$$
if $u\in \Wsp(U)$, $+\infty$ otherwise in $L^p(U)$.
The quantities $\theta$ and $\beta$ represent respectively the limit 
density and the limit normalized distribution of the 
set of points $\epsj\Lambda$ and can be explicitely calculated
in some cases (see \eqref{e:Rj/rj} and \eqref{e:unif-distrib} for the 
exact definitions, Remark~\ref{r:tight} for further comments, and 
Examples~\ref{ex:rescaled}-\ref{ex:diffeo2} where some cases are 
discussed in details).

Adding zero boundary conditions, $\Gamma$-convergence then implies the 
convergence for minimizers and minimum values of $(\FFepsj)_{j\in\N}$ 
to the respective quantities of $\FF$.
More generally, we can study the asymptotic behaviour of  
anisotropic variations of the Gagliardo seminorm 
(see subsection~\ref{s:generalizations}).

We remark that it is not our aim to establish a general abstract theory 
as Dal Maso did in the local case \cite{DM2},\cite{DM3}
(see \cite{Balzano} for related results in random settings), 
nor to consider the most general framework for homogenization as the 
one proposed by Nguetseng in the linear, local and non degenerate case 
\cite{Ng}; but rather we aim at giving explicit constructive results for a  
sufficiently broad class of fractional energies and non-periodic obstacles.

The main novelty of the paper is that we extend the asymptotic analysis of
obstacle probelms for local energies to non-local ones. In doing that 
we use intrinsic arguments and give a self-contained proof. We avoid 
extension techniques with which the original problem is transformed 
into a homogenization problem at the boundary by rewriting fractional 
seminorms as trace energies for local (but degenerate) functionals in 
one dimension higher (see  \cite{CMT}, \cite{Caf-Mel2}, \cite{Foc}). 
This is the well known harmonic extension procedure for $H^{1/2}$ functions; 
more recently it was proved to hold true also for the fractional laplacian, 
corresponding to $p=2$ and $s\in(0,1)$ above, by Caffarelli and Silvestre 
\cite{Caf-Silv}.  
Because of our approach, we are able to deal with non-local energies 
which can not be included in the previous frameworks.
Generalizations are likely to be done in several directions, the most 
immediate being the choice of the obstacle condition. For, we 
have confined ourselves to the basic bilateral zero obstacle 
condition in the scalar case only for the sake of simplicity, improvements 
to vector valued problems and unilateral obstacles seem to be at hand 
(see \cite{ANB}, \cite{Foc}).

Our analysis do not need the usual periodicity or almost periodicity 
assumptions for the distribution of obstacles. We deal with  
aperiodic settings defined after 
\emph{Delone set of points} (see subsection~\ref{s:tilings}), the set 
$\Lambda$ introduced above. 
No regularity or simmetry conditions are imposed on $\Lambda$, only 
two simple geometric properties are assumed: \emph{discreteness} and 
\emph{homogeneity}. They turn out to be physically reasonable conditions; 
as a matter of fact Delone sets have been introduced in $n$-dimensional 
mathematical cristallography to model many non-periodic structures such as 
quasicrystals (see \cite{Sen}). Essentially, these assumptions guarantee 
that points in $\Lambda$ can neither cluster nor be scattered away. 

Furthermore, we show that with minor changes the same tools are suited 
also to deal with some random settings. In particular, we deal with
obstacles with random sizes and shapes located on points of a standard
lattice in $\Rn$, a setting introduced by Caffarelli and Mellet 
\cite{Caf-Mel1}, \cite{Caf-Mel2}; and consider also 
homothetics random copies of a given obstacle located on random 
lattices following Blanc, Le Bris and Lions \cite{BLBL1}, 
\cite{BLBL2} (for more details see section~\ref{s:random}).

As a byproduct, our approach yields also an intrinsic proof of 
the results first obtained by \cite{Caf-Mel2} that avoids the  
extension techniques in \cite{Caf-Silv}. A different proof using 
$\Gamma$-convergence methods, but still relying on those extension 
techniques, was given by the author in \cite{Foc}.

The key tools of our analysis are Lemmas~\ref{tecnico2} and \ref{joining} 
below. By means of these results we reduce the $\Gamma$-limit 
process to families of functions which are constants on suitable 
annuli surrounding the obstacle sets. 
Lemma~\ref{joining} is the counterpart in the current non-local framework 
of the joining lemma in varying domains for gradient energies on 
standard Sobolev spaces proven by Ansini and Braides \cite{ANB}.
It is a variant of an idea by De Giorgi in the setting of varying 
domains, on the way of matching boundary conditions by increasing 
the energy only up to a small error. 
As in the local case the proofs of Lemmas~\ref{tecnico2} and \ref{joining} 
exploit De Giorgi's slicing/averaging principle and the fact that 
Poincar\'e-Wirtinger inequalities are qualitatively invariant under 
families of biLipschitz mappings with equi-bounded Lipschitz constants. 

Despite this, the non-local behaviour of fractional energies 
introduces several additional difficulties into the problem: 
Lemmas~\ref{tecnico2} and \ref{joining} do not follow 
from routine modifications of the arguments used in the local case. 
New ideas have to be worked out mainly to control the long-range 
interaction terms.
A major role in doing that is played 
by the counting arguments in Proposition~\ref{p:voroprop2},
Hardy inequality (see Theorem~\ref{hardy}), and the estimates 
on singular kernels in Lemma~\ref{Adams}.

The paper is organized as follows.
In section~\ref{s:prel} we list the necessary prerequisites on
$\Gamma$-convergence, Sobolev-Slobodeckij spaces and variational 
capacities giving precise references for those not proved.
Section~\ref{s:determ} is devoted to the exact statement and the proof 
of the homogenization result for deterministic distribution of 
obstacles. To avoid unnecessary generality we deal with the model 
case of fractional seminorms. Generalizations to anisotropic 
kernels are postponed to subsection~\ref{s:generalizations}.
The ideas of section~\ref{s:determ} are then used in section~\ref{s:random} 
to deal with the two different random settings mentioned before.
Finally, we give the proof of an elementary technical result, though 
instrumental for us, in Appendix~\ref{conti}.

\section{Preliminaries and Notations}\label{s:prel}

\subsection{Basic Notations}
We use standard notations for Lebesgue and Hausdorff measures, 
and for Lebesgue and Sobolev function spaces. 

The Euclidean norm in $\Rn$ is denoted by $|\cdot|$, the maximum 
one by $|\cdot|_\infty$. $B_r(x)$ stands for the Euclidean ball 
in $\Rn$ with centre $x$ and radius $r>0$, and we write simply 
$B_r$ in case $x=\underline{0}$. As usual $\omega_n:=\Ln(B_1)$.

Given a set $E\subset\R^n$ its complement will be indifferently 
denoted by $E^c$ or $\R^n\setminus E$.
Its interior and closure are denoted by $\mathrm{int(E)}$ and 
$\overline{E}$, respectively.
Given two sets $E\subset\subset F$ in $\Rn$, a 
\emph{cut-off function between E and F} is any 
$\varphi\in \Lip(\Rn,[0,1])$ such that $\varphi|_{\overline{E}}\equiv 1$, 
$\varphi|_{\Rn\setminus F}\equiv 0$, and 
$\Lip(\varphi)\leq 1/\dist(E,\partial F)$. 

Given an open set $A\subseteq\Rn$ the collections of its open, Borel 
subsets are denoted by $\AA(A)$, $\BB(A)$, respectively.
The diagonal set in $\Rn\times\Rn$ is denoted by $\Delta$, and 
for every $\delta>0$ its open $\delta$-neighborhood by
$\Delta_\delta:=\{(x,y)\in \Rn\times\Rn:\, |x-y|<\delta\}$.
Accordingly, for any set $E\subseteq\Rn$ and for any $\delta>0$ 
\begin{equation}\label{e:vdelta}
E_\delta:=\{x\in\Rn:\,\dist(x,E)<\delta\},\quad\quad
E_{-\delta}:=\{x\in E:\,\dist(x,\partial E)>\delta\}.
\end{equation}
In the following $U$ will always be an open and connected 
subset of $\Rn$ whose boundary is Lipschitz regular.

In several computations below the letter $c$ generically denotes a
positive constant. We assume this convention since it is not 
essential to distinguish from one specific constant to another, 
leaving understood that the constant may change from line to line.
The parameters on which each constant $c$ depends will be 
explicitely highlighted.

\subsection{$\Gamma$-convergence}
We recall the notion of $\Gamma$-convergence introduced by De Giorgi
in a generic metric space $(X,d)$ endowed with the topology induced
by $d$ (see \cite{DM},\cite{B3}).
A sequence of functionals $F_j:X\to [0,+\infty]$
{\it $\Gamma$-converges} to a functional $F:X\to [0,+\infty]$ in $u\in X$,
in short $F(u)=\Gamma\hbox{-}\lim_{j}F_j(u)$,
if 
the following two conditions hold:

{(i)} ({\it liminf inequality}) $\forall\  (u_j)_{j\in\N}$
converging to $u$ in $X$,
we have $\liminf_jF_{j}(u_j)\ge F(u)$;

{(ii)} ({\it limsup inequality}) $\exists$ $(u_j)_{j\in\N}$ converging 
to $u$ in $X$ such that $\limsup_jF_{j}(u_j)\le F(u)$.

\noindent We say that $F_j$ \emph{$\Gamma$-converges} to $F$
(or $F$= $\Gamma$-lim$_{j}F_j$) if
$F(u)=\Gamma\hbox{-}\lim_{j}F_j(u)$ $\forall u\in X$.
We may also define the \emph{lower} and \emph{upper $\Gamma$-limits} as
$$
\Gamma\hbox{-}\limsup_{j}F_j(u)=\inf\{\limsup_{j}
F_j(u_j):\  u_j\to u\},
$$
$$
\Gamma\hbox{-}\liminf_{j}F_j(u)=\inf\{\liminf_{j}
F_j(u_j):\  u_j\to u\},
$$
respectively, so that conditions (i) and (ii) are equivalent to
$\Gamma$-limsup$_j F_j(u)=\Gamma$-liminf$_j F_j(u)=F(u)$.
Moreover, the functions $\Gamma$-limsup$_j F_j$ and
$\Gamma$-liminf$_j F_j$ are lower semicontinuous.

One of the main reasons for the introduction of this notion is explained
by the following fundamental theorem (see \cite[Theorem 7.8]{DM}).

\begin{theorem}\label{min}
Let $F=\Gamma$-$\lim_{j}F_j$, and assume there exists
a compact set $K\subset X$  such that
$\inf_X F_j=\inf_K F_j$ for all $j$. Then there exists
$\min_X F =\lim_{j}\inf_X F_j$. Moreover,
if $(u_j)_{j\in\N}$ is a converging sequence such that
$\lim_j F_{j}(u_j)=\lim_j\inf_X F_{j}$
then its limit is a minimum point for $F$.
\end{theorem}

\subsection{Non-periodic tilings}\label{s:tilings}

In the ensuing sections we will deal with a general framework 
extending the usual periodic setting. 
The partition of $\Rn$ we consider is obtained via the 
\emph{Vorono\"i tessellation} related to a fixed 
\emph{Delone set of points} $\Lambda$.
We refer to the by now classical book of M. Senechal 
\cite{Sen} for all the relevant results.

\begin{definition}\label{voro}
A point set $\Lambda\subset\Rn$ is a {\rm Delone (or Delaunay) set} 
if it satisfies
\begin{itemize}
\item[(i)] {\rm Discreteness}: there exists $r>0$ such that for all 
$x,y\in\Lambda$, $x\neq y$, $|x-y|\geq 2r$;

\item[(ii)] {\rm Homogeneity or Relative Density}: 
there exists $R>0$ such that 
$\Lambda\cap B_R(x)\neq\emptyset$ for all $x\in\Rn$.
\end{itemize}
\end{definition}
It is then easy to show that $\Lambda$ is countably infinite. Hence,
from now on we use the notation $\Lambda=\{\xii\}_{\ii\in\Z^n}$.
By the very definition the quantities
\begin{equation}\label{d:rLRL}
\rL:=\frac 12\inf\{|x-y|:\,x,y\in\Lambda,\,x\neq y\},\quad
\RL:=\inf\{R>0:\,\Lambda\cap B_R(x)\neq \emptyset\;\;\forall x\in\Rn\}
\end{equation}
are finite and strictly positive; $\RL$ is called the 
\emph{covering radius} of $\Lambda$.
\begin{definition}
Let $\Lambda\subset\Rn$ be a Delone set, 
the {\rm Vorono\"i cell} of a point $\xii\in\Lambda$ is the set of 
points 
$$
V^{\ii}:=\{y\in\Rn:\,|y-\xii|\leq|y-\xkk|,\text{ for all }\ii\neq\kk\}.
$$
The {\rm Vorono\"i tessellation} induced by $\Lambda$ is the partition 
of $\Rn$  given by $\{V^{\ii}\}_{\ii\in\Z^n}$.
\end{definition}
In the following proposition we collect several interesting properties 
of Vorono\"i tessellations (see \cite[Propositions 2.7, 5.2]{Sen}).
\begin{proposition}\label{p:voroprop}
Let $\Lambda\subset\Rn$ be a Delone set and $\{V^{\ii}\}_{\ii\in\Z^n}$ its
induced Vorono\"i tessellation, then 
\begin{itemize}

\item[(i)] the $V^{\ii}$'s are convex polytopes fitting together 
along whole faces, and have no interior points in common;

\item[(ii)] if $V^{\ii}$ and $V^{\kk}$ share a vertex $z$, then $\xii$ 
and $\xkk$ lie on 
$\partial B_\rho(z)$ with $\Lambda\cap B_\rho(z)=\emptyset$, $\rho\leq\RL$. 

\end{itemize}
Hence, $\{V^{\ii}\}_{\ii\in\Z^n}$ is a {\rm tiling}, i.e. the $V^{\ii}$'s are 
closed, have no interior points in common and $\cup_{\ii}V^{\ii}=\Rn$.
More precisely, 
\begin{itemize}

\item[(iii)]  $\{V^{\ii}\}_{\ii\in\Z^n}$ is a {\rm normal} tiling: 
for each tile $V^{\ii}$ we have $B_{\rL}(\xii)\subseteq V^{\ii}\subseteq
\overline{B}_{\RL}(\xii)$; 

\item[(iv)] $\{V^{\ii}\}_{\ii\in\Z^n}$ is a {\rm locally finite} tiling: 
$\#(\Lambda\cap B_\rho(x))<+\infty$ for all $x\in\Rn$, $\rho>0$.
\end{itemize}
\end{proposition}
Further properties that will be repeatedly used in our 
analysis are summarized below. We omit their proofs 
since they are justified by elementary counting arguments. 
For any $A\in\AA(\Rn)$ we set
\begin{equation}\label{e:indices}
\IL(A):=\{\ii\in\Z^n:\,V^{\ii}\subseteq A\},\quad\quad 
\IIL(A):=\{\ii\in\Z^n:\,V^{\ii}\cap \partial A\neq\emptyset\}.
\end{equation}
\begin{proposition}\label{p:voroprop2}
Let $\Lambda\subset\Rn$ be a Delone set and $\{V^{\ii}\}_{\ii\in\Z^n}$ its
induced Vorono\"i tessellation. Then,
\begin{equation}\label{e:counting}
\omega_n\rL^n\#(\IL(A))\leq \Ln(A),\quad
\omega_n\rL^n\#(\IIL(A))\leq (\partial A)_{\RL},\quad
\Ln\left(A\setminus\cup_{\IL(A)}V^{\ii}\right)\leq (\partial A)_{\RL}.
\end{equation}
In particular, there exists a constant $c=c(n)>0$ such that 
for every $\ii\in\Z^n$, $m\in\N$ it holds
\begin{equation}\label{e:counting2}
\#\{\kk\in\IL(A):\,m\rL<|\xii-\xkk|_\infty\leq(m+1)\rL\}\leq 
c\, m^{n-1}.
\end{equation}
\end{proposition}

\subsection{Sobolev-Slobodeckij spaces}\label{sobslob} 

Let $A\subseteq\R^n$ be any bounded open Lipschitz set, 
$p\in(1,+\infty)$, $s\in(0,1)$ and $ps\in(1,n)$, by $\Wsp(A)$ we denote 
the usual Sobolev-Slobodeckij space, or Besov space $B^s_{p,p}(A)$. 
The space is Banach if equipped with the norm 
$\|u\|_{\Wsp(A)}=\|u\|_{L^p(A)}+|u|_{\Wsp(A)}$, where
$$
|u|_{\Wsp(A)}^p:=\int_{A\times A}\frac{|u(x)-u(y)|^p}{|x-y|^{n+sp}}dxdy\,.
$$
We will use several properties of fractional Sobolev spaces, 
giving precise references for those employed in the sequel 
in the respective places mainly referring to \cite{AD} and 
\cite{Tr0}. 

In the indicated ranges for the parameters $p$, $s$ it turns out 
that $\Wsp$ is a reflexive space (see \cite[Thm 4.8.2]{Tr0}), 
Sobolev embedding and Sobolev-Gagliardo-Nirenberg inequality hold
(see \cite[Chapter V]{AD}), and traces are well defined 
(see \eqref{e:traces} below).
We remark that the restriction on $s$ are necessary, since
otherwise $\Wsp(A)$ contains only constant functions if $s\geq 1$, 
while for $ps<1$ traces are not well defined (see also 
\eqref{e:traces} below).
The exclusion of the other cases is related to capacitary issues 
(see subsection~\ref{cap}).

Poincar\`e-Wirtinger and Poincar\`e inequalities in fractional 
Sobolev spaces are instrumental tools in the sequel. Thus,
we state explicitely those results in the form we need.
Their proof clearly follows from the usual argument by contradiction 
once $\Wsp$ is reflexive and is endowed with a trace operator.
\begin{theorem}\label{PW}
Let $n\geq 1$, $p\in(1,+\infty)$, and $s\in(0,1)$. 
Let $A\subset\Rn$ be a bounded, connected open set, and $O$ 
any measurable subset of $A$ with $\L^n(O)>0$. Then for any function 
$u\in\Wsp(A)$, 
  \begin{equation}
    \label{e:PW}
\|u-u_O\|^p_{L^p(A)}\leq c_{PW}
|u|^p_{\Wsp(A)},    
  \end{equation}
for a constant $c_{PW}=c_{PW}(n,p,s,O,A)$.
Moreover, for any $u\in\Wspo(U)$ we have 
 \begin{equation}
    \label{e:P}
\|u\|^p_{L^p(A)}\leq c_{P}
|u|^p_{\Wsp(A)},
 \end{equation}
for a constant $c_P=c_{P}(n,p,s,A)$.
\end{theorem}
\begin{remark}\label{r:PWscaled}
Let $(\Phi_t)_{t\in\mathscr{T}}$ be a family of biLipschitz maps 
on $A$ with $\sup_{ \mathscr{T}}(\Lip(\Phi_t)+\Lip(\Phi^{-1}_t))<+\infty$,
then a simple change of variables implies that the constants 
 $c_{PW}(n,p,s,\Phi_t(O),\Phi_t(A))$ and $c_{P}(n,p,s,\Phi_t(A))$ 
are uniformly-bounded in $t$. 
In particular, 
a scaling argument and H\"older inequality 
yield for any $z\in\R^n$ and $r>0$ and for some $c=c(n,p,s,O,A)>0$
  \begin{equation}\label{e:PWscaled}
\|u-u_{z+r O}\|^p_{L^p(z+r A)}\leq c\,r^{sp}|u|^p_{\Wsp(z+r A)}.    
  \end{equation}
A similar conclusion holds for Poincar\'e inequality \eqref{e:P}. 
\end{remark}
Next, we recall the fractional version of Hardy inequality (see 
\cite[Theorem 4.3.2/1, Remark 2 pp. 319-320]{Tr0} and \cite{Tr} 
for further comments). To this aim we introduce the space
$\tWsp(B_1):=\{u\in\Wsp(\R^n):\,\mathrm{spt}u\subset\overline{B}_1\}$.
It is clear that $C^\infty_0(B_1)$ is dense in $\tWsp(B_1)$ for any 
$p\in(1,+\infty)$ and that $\tWsp(B_1)\subseteq\Wsp_0(B_1)$. 
The latter inclusion is strict if $s-1/p\in\N$, and more precisely 
it holds
\begin{eqnarray}\label{e:traces}
\Wsp_0(B_1)=
\begin{cases}
\tWsp(B_1) &  \text{ if } s>1/p-1,\, s-1/p\notin\N \\
\Wsp(B_1) &   \text{ if } s\in(0,1/p].
\end{cases}
\end{eqnarray}
\begin{theorem}\label{hardy}
There exists a constant $c=c(n,p,s)$ such that for every 
$u\in \tWsp(B_1)$
we have 
$$
\int_{B_1}\frac{|u(x)|^p}{\dist(x,\partial B_1)^{sp}}\dx\leq
c\left(|u|_{\Wsp(B_1)}^p+\|u\|_{L^p(B_1)}^p\right).
$$
\end{theorem}
\begin{remark}\label{Hardyscaled}
The usual scaling argument, Poincar\'e inequality \eqref{e:P}  
and Theorem~\ref{hardy} then yield
\begin{equation}
  \label{e:Hardy2}
  \int_{B_r}\frac{|u(x)|^p}{\dist(x,\partial B_r)^{sp}}\dx\leq
\cnps|u|_{\Wsp(B_r)}^p
\end{equation}
for every $r>0$ and $u\in \tWsp(B_r)$, for a constant $\cnps=c(n,p,s)$ 
independent from $r$.
\end{remark}

\subsection{Fractional capacities}\label{cap}

We recall the notion of variational capacity for fractional Sobolev
spaces and prove some properties relevant in the developments below.
Those properties, straightforward in the local case, require more work
in the non-local one. Let $p\in(1,+\infty)$ and $s\in(0,1)$ be given 
as before, for any set $T\subseteq\Rn$ define
\begin{equation}\label{e:capsp}
\csp(T):=\inf_{\{A\in\AA(\Rn):\,A\supseteq T\}}
\inf\left\{|u|^p_{\Wsp(\Rn)}:\,
u\in\Wsp(\Rn),\,u\geq 1\,\Ln \text{ a.e. on } A\right\},
\end{equation}
with the usual convention $\inf\emptyset=+\infty$.
The set function in \eqref{e:capsp} turns out to be a Choquet capacity
(see \cite[Chapter V]{AD}).
Recall that a property holds $\csp$ \emph{quasi everywhere}, 
in short $\csp$ q.e. on $A$, if it holds up to a set of $\csp$ zero. 
In particular, any function $u$ in $\Wsp(A)$, $A\in\AA(\Rn)$, 
has a \emph{precise representative} $\tilde{u}$ defined $\csp$ 
q.e. and the following formula holds (see \cite[Proposition 5.3]{AD})
\begin{equation}
  \label{e:capaltern}
  \csp(T):=\inf\left\{|w|_{\Wsp(\Rn)}^p:\,
w\in\Wsp(\Rn),\,\tilde{w}\geq 1\text{ q.e. on } T\right\}.
\end{equation}
Coercivity of the fractional norm is ensured only in the 
$L^{p^\ast}(\Rn)$ topology, $p^\ast:=np/(n-sp)$ is the 
Sobolev exponent relative to $p$ and $s$ (see \cite[Chapter V]{AD}).
Thus, a minimizer for the capacitary problem exists in the 
homogeneous space  
$\Ksp(\Rn)=\{u\in L^{p^\ast}(\Rn):\, |u|_{\Wsp(\Rn)}<+\infty\}$.
Uniqueness is guaranteed by the strict convexity of the 
fractional energy, thus the minimizer of \eqref{e:capsp}, 
\eqref{e:capaltern} will be denoted by $u^T$ and called 
the \emph{capacitary potential} for $T$.

\begin{remark}
If $ps>n$ points have positive capacity and $\Wsp$ is embedded 
into $C^0$. In this case it is well known that the homogenized 
obstacle problem turns out to be trivial, this is the reason 
why we required $ps<n$.
The borderline case $ps=n$ deserves an analysis similar to that 
we will perform but different in some details, so that   
its study is not dealt with in this paper (see for instance 
\cite{CM} and \cite{Sigalotti} in the local framework).
\end{remark}
Let us prove that set inclusion induces an ordering among 
capacitary potentials. As a byproduct we also show that admissible
functions in the capacitary problem can be taken with values in $[0,1]$. 
\begin{lemma}\label{l:cap}
If $T\subseteq F$, then $0\leq u^T\leq u^F\leq 1$ $\Ln$ a.e. on $\Rn$.
\end{lemma}
\begin{proof}
First, take note that for all $u$, $v\in L^1_{\mathrm{loc}}(\Rn)$ we have
\begin{eqnarray}\label{e:trunc}
\begin{array}{c}
|(u\vee v)(x)-(u\vee v)(y)|\leq|u(x)-u(y)|\vee|v(x)-v(y)|,\\
|(u\wedge v)(x)-(u\wedge v)(y)|\leq|u(x)-u(y)|\vee|v(x)-v(y)|.
\end{array}
\end{eqnarray}
In particular, uniqueness of the capacitary potential $u^T$ and by 
choosing $v\equiv 0$ in \eqref{e:trunc}$_1$ and 
$v\equiv 1$ in \eqref{e:trunc}$_2$ imply that 
$0\leq u^T\leq 1$ $\Ln$ a.e. on $\Rn$ for any subset $T\subset\Rn$.

Moreover, \eqref{e:trunc} yields that $u^T\vee u^F$ and 
$u^T\wedge u^F\in\Ksp(\Rn)$. Set $U_T=\{u^T\leq u^F\}$ and 
$U_F=\{u^F<u^T\}$, and assume by contradiction that $\Ln(U_F)>0$.
By taking $u^T\vee u^F$ as test function in the minimum problem 
$\csp(F)$ and recalling the strict minimality of $u^F$, we infer
$|u^F|^p_{\Wsp(\Rn)}<|u^T\vee u^F|^p_{\Wsp(\Rn)}$. An easy 
computation then leads to 
$$
|u^F|^p_{\Wsp(U_F)}+2\int_{U_T\times U_F}\frac{|u^F(x)-u^F(y)|}{|x-y|^{n+sp}}dxdy
<
|u^T|^p_{\Wsp(U_F)}+2\int_{U_T\times U_F}\frac{|u^T(x)-u^F(y)|}{|x-y|^{n+sp}}dxdy.
$$
The latter inequality can be used to estimate the fractional norm of  
$u^T\wedge u^F$ as follows
\begin{eqnarray}\label{e:dupalle}
\lefteqn{|u^T\wedge u^F|^p_{\Wsp(\Rn)}<|u^T|^p_{\Wsp(\Rn)}}\\&&
+2\int_{U_T\times U_F}\frac{|u^T(x)-u^F(y)|^p+|u^F(x)-u^T(y)|^p
-|u^T(x)-u^T(y)|^p-|u^F(x)-u^F(y)|^p}{|x-y|^{n+sp}}dxdy.\notag
\end{eqnarray}
We claim that the second term on the rhs in \eqref{e:dupalle} 
is non-positive, this would imply that the strict inequality 
sign holds above, which in turn would give a contradiction since 
$u^T\wedge u^F$ is a test function for the capacitary problem 
related to $T$.

To conclude consider the auxiliary function
$\psi(s,t):=|s-u^T(y)|^p+|u^T(x)-t|^p-|s-t|^p-|u^T(y)-u^T(x)|^p$ 
and the set $H:=\{(s,t):\,s\geq u^T(x),\;t<u^T(y)\}$. 
Elementary calculations yield that $\max_H\psi\leq 0$.
Take note that the numerator of the integrand of the second term 
in the rhs of \eqref{e:dupalle} equals to 
$\psi(u^F(x),u^F(y))$ and $u^F(x)\geq u^T(x)$ $\Ln$ a.e. $x\in U_T$, 
and $ u^F(y)<u^T(y)$ $\Ln$ a.e. $y\in U_F$. 
\end{proof}
In the sequel we are interested into relative capacities,
for which we introduce two different notions. The first one is 
useful in the $\Gamma$-liminf inequality, the second in the 
$\Gamma$-limsup inequality, respectively. 
For every $0<r\leq R$ set
$$
\csp(T,B_R;r):=\inf\left\{|w|_{\Wsp(B_R)}^p:\,
w\in\Wsp(\Rn),\, w=0 \text{ on }\Rn\setminus \overline{B}_r,\,
\tilde{w}\geq 1\text{ q.e. on } T\right\},
$$
and
$$
\CSP(T,B_R):=\inf\left\{|w|_{\Wsp(\Rn)}^p:\,
w\in\Wsp(\Rn),\, w=0 \text{ on }\Rn\setminus \overline{B}_R,\,
\tilde{w}\geq 1\text{ q.e. on } T\right\}.
$$
To prove the convergence of relative capacities to the global one
we introduce some notation to simplify the calculations below: 
for any ${\mathcal L}^{n}$-measurable function $w$
and any ${\mathcal L}^{n\times n}$-measurable set $E\subseteq U\times U$ 
consider the \emph{locality defect} of the $\Wsp$ seminorm 
$$
\defectsp(w,E):=\int_{E}\frac{|w(x)-w(y)|^p}{|x-y|^{n+sp}}dxdy\,.
$$
The terminology is justified since given two disjoint subdomains 
$A,B\subseteq U$, it holds
\begin{equation}\label{locdef}
|w|_{\Wsp(A\cup B)}^p=|w|_{\Wsp(A)}^p+|w|_{\Wsp(B)}^p+
2\defectsp(w,A\times B).
\end{equation}
In particular, $\defectsp(w,A\times A)=|w|_{\Wsp(A)}^p$.

We are now in a position to prove the claimed converge result for 
relative capacities. 
Actually, we show uniform convergence for families of equi-bounded 
sets. A different argument yielding pointwise convergence will
be exploited in the generalizations of subsection \ref{s:generalizations}
(see Lemma~\ref{loccap2}).
The latter is sufficient for the proof of Theorem~\ref{main}; 
the advantage of the approach below is that it can be carried over 
straightforward to the case of obstacles with random sizes and shapes 
for which uniform convergence is necessary (see subsection~\ref{caf-mell}).
\begin{lemma}\label{loccap}
For all $\rho>0$ it holds 
\begin{equation}\label{e:cap1}
\lim_{r\to+\infty}\sup_{T\subseteq B_\rho}|\CSP(T,B_r)-\csp(T)|=0.
\end{equation}
Moreover, there exists a constant $\cnps=c(n,s,p)$ such that 
for all $0<\rho<r<R$
\begin{equation}\label{e:cap2}
\sup_{T\subseteq B_\rho}\left(\csp(T)-\csp(T,B_R;r)\right)\leq
\frac{\cnps r^{sp}}{(R-r)^{sp}}\CSP(B_\rho,B_r).
\end{equation}
In addition, if $R(r)/r\to+\infty$ as $r\to+\infty$ we have
\begin{equation}\label{e:cap3}
\lim_{r\to+\infty}\sup_{T\subseteq B_\rho}|\csp(T)-\csp(T,B_{R(r)};r)|=0.
\end{equation}
\end{lemma}
\begin{proof}
It is clear from the very definitions that $(0,+\infty)\ni r\to\CSP(T,B_r)$ 
is monotone decreasing, and moreover that $\csp(T,B_R;r)\leq\CSP(T,B_r)$,
$\csp(T)\leq \CSP(T,B_r)$.

Let us first prove \eqref{e:cap1}. To this aim take
$u^T$, $u^{B_\rho}\in\Ksp(\Rn)$ the capacitary potentials of $T$ and $B_\rho$, 
respectively; by Lemma~\ref{l:cap} we have $0\leq u^T\leq u^{B_\rho}\leq 1$
$\Ln$ a.e. on $\Rn$. In addition, $u^{B_\rho}$ is radially symmetric and 
decreasing (to zero) since the fractional norm is stricly convex, 
rotation invariant and decreasing under radial rearrangements 
(for the last result see \cite[Section 9]{Alm-Lieb}).

With fixed $\delta>0$ consider the Lipschitz map
$\psi_\delta(t):=\frac{t-\delta}{1-\delta}\vee 0$, 
$\Lip(\psi_\delta)\leq(1-\delta)^{-1}$, and set
$w_\delta(x):=\psi_\delta(u^T(x))$. 
Up to $\Ln$ negligible sets, $\{w_\delta>0\}=\{u^T>\delta\}\subseteq
\{u^{B_\rho}>\delta\}\subseteq B_{R_\delta}$, for some $R_\delta\to+\infty$
as $\delta\to 0^+$. Then 
$w_\delta\in\Wsp(\Rn)$ with 
$$
|w_\delta|^p_{\Wsp(\Rn)}\leq\frac{1}{(1-\delta)^p}|u^T|^p_{\Wsp(\Rn)}=
\frac{1}{(1-\delta)^p}\csp(T),\quad
\|w_\delta\|_{L^p(\Rn)}\leq\frac{1}{1-\delta}\|u^T\|_{L^p(B_{R_\delta})}.
$$
Moreover, $\tilde{w}_\delta\geq 1$ q.e. on $T$.
Since $\csp(\cdot)$ is an increasing set function, we infer
$$
0\leq\CSP(T,R_\delta)-\csp(T)\leq\left(\frac{1}{(1-\delta)^p}-1\right)
\csp(T)\leq\left(\frac{1}{(1-\delta)^p}-1\right)\csp(B_\rho).
$$
In conclusion, \eqref{e:cap1} follows from the monotonicity 
properties of $r\to\CSP(T,B_r)$.

To prove \eqref{e:cap2} take note that any admissible function $u$ 
for the minimum problem defining $\csp(T,B_R;r)$ is admissible 
for the one defining $\csp(T)$, too. Furthermore, for some constant
$\cnps=c(n,s,p)$ it holds 
\begin{eqnarray*}
\lefteqn{\csp(T)\leq|u|^p_{\Wsp(\Rn)}\stackrel{\eqref{locdef}}{=}
|u|^p_{\Wsp(B_R)}+|u|^p_{\Wsp(B_R^c)}+2\defectsp(u,B_R\times B_R^c)}\\&&
\stackrel{u|_{B_r^c}=0}{=}
|u|^p_{\Wsp(B_R)}+2\int_{B_r}dx\int_{B^c_R}\frac{|u(x)|^p}{|x-y|^{n+sp}}dy
\stackrel{(ii) \text{ Lemma } \ref{Adams}}{\leq}
|u|^p_{\Wsp(B_R)}+\cnps\int_{B_r}\frac{|u(x)|^p}{\dist^{sp}(x,\partial B_R)}dx\\
&&\leq|u|^p_{\Wsp(B_R)}+\frac{\cnps}{(R-r)^{sp}}\int_{B_r}|u(x)|^pdx
\leq|u|^p_{\Wsp(B_R)}+\frac{\cnps r^{sp}}{(R-r)^{sp}}|u|^p_{\Wsp(B_r)}.
\end{eqnarray*}
In the last inequality we used the scaled version of 
Poincar\`e inequality \eqref{e:P} as follows from Remark~\ref{r:PWscaled}.
By passing to the infimum on the admissible test functions we infer
$$
\csp(T)-\csp(T,B_R;r)\leq\frac{\cnps r^{sp}}{(R-r)^{sp}}\csp(T,B_R;r)\leq
\frac{\cnps r^{sp}}{(R-r)^{sp}}\CSP(T,B_r).
$$
We deduce statement \eqref{e:cap2} since $\CSP(\cdot,B_r)$
is a monotone increasing set function.

Eventually, \eqref{e:cap3} follows at once from \eqref{e:cap1},
\eqref{e:cap2}, and the fact that $\csp(T,B_R;r)\leq\CSP(T,B_r)$. 
\end{proof}
\begin{remark}
Clearly estimate \eqref{e:cap2} blows up for $r=R$. In such a case 
by using Hardy inequality one can only prove that 
$\csp(T)\leq(1+c)\csp(T,B_R,R)$ for some $c=c(n,s,p)>0$. 
\end{remark}
\begin{remark}\label{r:caplodef}
If $\xi_r$ is a $(1/r)$-minimizer for $\CSP(T,B_r)$ and 
$R(r)/r\to+\infty$ as $r\to+\infty$, then
$$  
 \lim_{r\to+\infty}\defectsp(\xi_r,B_{R(r)}\times {B}_{R(r)}^c)=0.
$$
Indeed, being $\xi_r$ admissible for the problem defining
$\csp(T,B_R;r)$, for all $r<R$, with 
$$
\csp(T,B_R;r)\leq|\xi_r|^p_{\Wsp(B_R)}\leq
|\xi_r|^p_{\Wsp(\Rn)}\leq \CSP(T,B_r)+\frac 1r,
$$ 
from \eqref{e:cap3} we conclude. 
\end{remark}

\section{Deterministic Setting}\label{s:determ}

\subsection{Statement of the Main Result}

Consider Delone sets $\Lambda_j=\{\xiij\}_{\ii\in\Z^n}$, and let 
$\Rj:=R_{\Lambda_j}$, $\Ieps(A):={\mathcal I}_{\Lambda_j}(A)$,
$\IIeps(A):={\mathscr I}_{\Lambda_j}(A)$, for all $A\in\AA(U)$
(see \eqref{e:indices} for the definition of $\Ieps$, $\IIeps$). 
Fix $\rj\in(0,r_{\Lambda_j}]$, and assume that the $\rj$'s and 
$\Lambda_j$'s are such that 
\begin{equation}\label{e:Rj/rj}
\lim_j\rj=0,\quad (1\leq)\limsup_j(\Rj/\rj)<+\infty,
\end{equation}
\begin{equation}\label{e:lambdaj}
\lim_j\#\Ieps(U)\,\rj^n=\theta\in(0,+\infty),
\end{equation}
\begin{equation}\label{e:unif-distrib}
\mu_j:=\frac 1{\#\Ieps(U)}\sum_{\ii\in\Ieps(U)}\delta_{\xiij}\rightarrow
\mu:=\beta \Ln\res U\quad\quad w^*\hbox{-}C_b(U),
\end{equation}
for some $\beta\in L^1(U,[0,+\infty])$ with $\|\beta\|_{L^1(U)}=1$.
\begin{remark}
Condition \eqref{e:lambdaj} implies that $\rj\sim r_{\Lambda_j}$
since $\limsup_j\#\Ieps(U)\,r_{\Lambda_j}^n<+\infty$ by \eqref{e:counting}$_1$
and $\liminf_j\#\Ieps(U)\,r_{\Lambda_j}^n>0$  by \eqref{e:counting}$_2$
(see also Remark~\ref{r:triviallimit}).
\end{remark}
\begin{remark}\label{r:tight}
It is well known that the $w^*\hbox{-}C_b(U)$ convergence of 
$(\mu_j)_{j\in\N}$ to $\mu$ in \eqref{e:unif-distrib} can be restated as
\begin{equation}\label{e:unif-distrib2}
\mu_j(A)\to\mu(A)\quad
\text{ for all } A\in\AA(U) \text{ with } \mu(U\cap\partial A)=0.
\end{equation}
Proposition~\ref{p:voroprop2} and conditions 
\eqref{e:Rj/rj}, \eqref{e:lambdaj} imply that 
assumption \eqref{e:unif-distrib} is always satisfied 
up to subsequences. First, let us show that any $w^*\hbox{-}C_0^0(U)$ 
cluster point of the probability measures $(\mu_j)_{j\in\N}$ is 
absolutely continuous w.r.to $\Ln\res U$. 
For, let $\mu_j$ converge to $\mu$ $w^*\hbox{-}C_0^0(U)$, then
since $B_{\rj}(\xiij)\subseteq\Vij$, for 
every $\varphi\in C_0^0(U)$ and $\delta>0$ uniform continuity 
yields for $j$ sufficiently big 
$$
\int_U\varphi\,d\mu_j=\frac{1}{\#\Ieps(U)}\sum_{\Ieps(U)}\varphi(\xiij)
\leq\frac{1}{\#\Ieps(U)}\sum_{\Ieps(U)}\fint_{\Vij}\varphi\,dx+\delta
\leq\frac{1}{\omega_n\rj^n\#\Ieps(U)}\int_U|\varphi|\,dx
+\delta.
$$
By taking as test functions $\pm\varphi$ and first letting 
$j\to+\infty$ and then $\delta\to 0^+$ we infer
$$
\left|\int_U\varphi\,d\mu\right|\leq
\frac{1}{\omega_n\theta}\int_U|\varphi|\,dx.
$$
The latter inequality implies $\mu\ll\Ln\res U$.
Actually, since $\Vij\subseteq\overline{B}_{\Rj}(\xiij)$ 
arguing as above it follows $\beta\in L^\infty(U)$ with 
$(\liminf_j\rj/\Rj)^n\leq\omega_n\theta \beta(x)\leq 1$ 
for $\L^n$ a.e. $x\in U$. 

Furthermore, take $A^\prime$, $A\in\AA(U)$ such that 
$A^\prime\subset\subset A\subset\subset U$, 
then for $j$ sufficiently big we have
\begin{multline*}
\mu_j(\overline{A})=\frac{\#\{\ii\in\Ieps(U):\,\xiij\in\overline{A}\}}
{\#\Ieps(U)}\geq\frac{\#\Ieps(A^\prime)+\#\IIeps(A^\prime)}{\#\Ieps(U)}\\
=1-\frac{\#\Ieps(U\setminus\overline{A^\prime})}{\#\Ieps(U)}
\stackrel{\eqref{e:counting}_1,\eqref{e:counting}_2}{\geq} 
1-\frac{\Ln(U\setminus\overline{A^\prime})}
{\left(\frac{r_{\Lambda_j}}{R_{\Lambda_j}}\right)^n\Ln(U)
-(\partial U)_{R_{\Lambda_j}}}.
\end{multline*}
Hence, equi-tightness of $(\mu_j)_{j\in\N}$ follows from
$$
\liminf_j\mu_j(\overline{A})\geq 1-\limsup_j
\left(\frac{R_{\Lambda_j}}{r_{\Lambda_j}}\right)^n
\frac{\Ln(U\setminus\overline{A^\prime})}{\Ln(U)}.
$$
Thus, Prokhorov theorem gives the w$^\ast$-$C_b(U)$ compactness 
in \eqref{e:unif-distrib} up to subsequences. 
\end{remark}

With fixed a bounded set $T$, for all $j\in\N$ 
define the \emph{obstacle set} $\Kj\subseteq\Rn$ by $\Kj=\cup_{\ii\in\Z^n}\Kij$ 
where 
\begin{equation}\label{e:Kij}
\Kij:=\xiij+\lambdaj T,\quad \text{ and } \lambdaj:=\rj^{n/(n-sp)}.
\end{equation}
Take note that $\Kij\subseteq\Vij$ for all $\ii\in\Z^n$ and $j\in\N$.

Consider the functionals $\FFepsj:L^p(U)\to[0,+\infty]$ defined by
\begin{eqnarray}  \label{e:fapprox}
  \FFepsj(u)=
  \begin{cases}
\displaystyle{|u|_{\Wsp(U)}^p} & \text{ if } u\in \Wsp(U),\,
\tildeu=0\,\,  \csp \text{ q.e. on } \Kj\cap U\\
+\infty & \text{ otherwise. }
  \end{cases}
\end{eqnarray}
\begin{theorem}\label{main}
Let $U\in\AA(\Rn)$ be bounded and connected with Lipschitz regular 
boundary; and assume that \eqref{e:Rj/rj}-\eqref{e:unif-distrib} are 
satisfied.

The sequence $(\FFepsj)_{j\in\N}$ $\Gamma$-converges in the $L^p(U)$
topology to $\FF:L^p(U)\to[0,+\infty]$ defined by
\begin{equation}\label{e:Glimit}
\FF(u)=|u|^p_{\Wsp(U)}+\theta\,\csp(T)\int_{U}|u(x)|^p\beta(x)\,dx
\end{equation}
if $u\in \Wsp(U)$, $+\infty$ otherwise in $L^p(U)$.
\end{theorem}
\begin{remark}\label{r:triviallimit}
If $\theta\in\{0,+\infty\}$ simple comparison arguments show 
that the conclusions of Theorem~\ref{main} still hold true,
though the $\Gamma$-limit is trivial in both cases.
\end{remark}
Let us show some examples of sets of points included in the 
framework above.   
In the sequel $(\epsj)_{j\in\N}$ will always denote a positive 
infinitesimal sequence.
\begin{example}\label{ex:rescaled}
Given a Delone set of points $\Lambda$ in $\Rn$ let 
$\Lambda_j:=\epsj\Lambda$, then $r_{\Lambda_j}=\epsj r_\Lambda$ 
and $R_{\Lambda_j}=\epsj R_\Lambda$. 
If $\rj\sim r_{\Lambda_j}$ assumptions \eqref{e:lambdaj} and 
\eqref{e:unif-distrib} hold true up to the extraction of 
a subsequence according to Remark~\ref{r:tight}, respectively. 
Several ways of generating Delone sets of points are discussed 
in \cite{Sen}.
\end{example}

More explicit examples can be obtained as follows
(see Examples~\ref{ex:diffeostoch1}, \ref{ex:diffeostoch2} 
for the stochastic versions).  
\begin{example}\label{ex:diffeo1}
Let $\Phi:\Rn\to\Rn$ be a diffeomorphism satisfying
$$
\|\nabla\Phi\|_{L^\infty(\Rn)}\leq M,\quad \text{and }\quad 
\inf_{\R^n}\det\nabla\Phi\geq\nu>0.
$$
Then the smallest eigenvalue of $\nabla^t\Phi\nabla\Phi$ is 
greater than $\nu M^{1-n}$, 
and thus for all $x,y\in\Rn$ 
$$
\nu M^{1-n}|x-y|\leq|\Phi(x)-\Phi(y)|\leq M|x-y|.
$$
Set $\Lambda_j=\{\Phi(\epsj\ii)\}_{\ii\in\Z^n}$, then 
$(\nu M^{1-n}/2)\epsj\leq r_{\Lambda_j}\leq R_{\Lambda_j}\leq M\epsj$.
An easy computation and \eqref{e:unif-distrib2} yield the 
$w^*\hbox{-}C_b(U)$ convergence of the measures $\mu_j$ in 
\eqref{e:unif-distrib} to $\mu=\beta\Ln\res U$ with
$$
\beta(x)=\left(\int_U\det\nabla\Phi^{-1}(x)\,dx\right)^{-1}
\det\nabla\Phi^{-1}(x).
$$
Eventually, if $\rj\in(0, r_{\Lambda_j}]$ with $\rj/\epsj\to\gamma>0$ 
we have $\theta=\gamma^n\int_U\det\nabla\Phi^{-1}(x)\,dx$.
\end{example}
\begin{example}\label{ex:diffeo2}
Let $\Phi$ be a diffeomorphism as in the previous example, 
and set $\Lambda_j=\{\epsj\Phi(\ii)\}_{\ii\in\Z^n}$. 
As before, we have
$(\nu M^{1-d}/2)\epsj\leq r_{\Lambda_j}\leq R_{\Lambda_j}\leq M\epsj$, 
by using \eqref{e:unif-distrib2} it can be checked that if 
$(\det\nabla\Phi^{-1})(\cdot/\epsj)_{j\in\N}$ converge to $g$ weakly $*$ 
in $L^\infty_{\mathrm{loc}}(\Rn)$, the measures $\mu_j$ in 
\eqref{e:unif-distrib} 
converge $w^*\hbox{-}C_b(U)$ to $\mu=\beta\Ln\res U$ with
$$
\beta(x)=\left(\int_Ug(x)\,dx\right)^{-1}g(x).
$$ 
By choosing $\rj\in(0, r_{\Lambda_j}]$ with $\rj/\epsj\to\gamma>0$,
we have $\theta=\gamma^n\int_Ug(x)\,dx$.

\end{example}


\subsection{Technical Lemmas}\label{s:technical}

To prove Theorem~\ref{main} we establish two technical results 
which are instrumental for our strategy. 
Roughly speaking we show that 
the $\Gamma$-limit can be computed on sequences of functions 
matching the values of their limit on suitable annuli surrounding 
the obstacle sets. 
To give a proof as much clear as possible we first work in an 
unscaled setting in Lemma~\ref{tecnico2}, and then turn to the 
framework of interest in Lemma~\ref{joining}. 
The method of proof is elementary and based on a clever slicing 
and averaging argument, looking for those zones where the energy 
does not concentrate.
The relevant property we prove is that the energetic error of 
the construction we perform is estimated by a local term: a measure.

\begin{lemma}\label{tecnico2}
Let $\Lambda$ be a Delone set of points. For any $m\in\N$, $m\geq 2$, 
$\rho\in(0,\rL/2)$ and $\ii\in\IL(U)$ let 
$\Apr_{\ii}=\xii+B_{\rho/m}\setminus\overline{B}_{\rho/m^2}$, 
$\A_{\ii}=\xii+B_\rho\setminus\overline{B}_{\rho/m^3}$,
and $\varphi_{\ii}(\cdot)=\varphi(\cdot-\xii)$, where $\varphi$ is a cut-off
function between $B_{\rho/m}\setminus\overline{B}_{\rho/m^2}$
and $B_\rho\setminus\overline{B}_{\rho/m^3}$.

Then there exists a constant $c=c(n,p,s)>0$ such that for any function
$u\in \Wsp(U)$, and any $\#\IL(U)$-tuple
of vectors $\{z_{\ii}\}_{\ii\in\IL(U)}$, $z_{\ii}\in\Rn$, the 
function 
$$
w(x)=\sum_{\ii\in\IL(U)}\varphi_{\ii}(x)z_{\ii}
+\left(1-\sum_{\ii\in\IL(U)}\varphi_{\ii}(x)\right)u(x)
$$ 
belongs to $\Wsp(U)$, $u=z_{\ii}$ on $\Apr_{\ii}$ and $w=u$ on $\Ac$, with 
 $\A:=\cup_{\ii\in\IL(U)}\A_{\ii}$; in addition
for every measurable set $E\subseteq U\times U$ it holds
\begin{eqnarray}\label{stima00}
\left|\defectsp(w,E)-\defectsp(u,E)\right|\leq
\cnps\left(\defectsp(u,U\times \A)+m^{2p}\rho^{-sp}
\sum_{\ii\in\IL(U)}\int_{\A_{\ii}}\left|u(y)-z_{\ii}\right|^pdy\right).
\end{eqnarray}
\end{lemma}
\begin{proof}
By construction $w=u$ on $\Ac$ and $u=z_{\ii}$ on $\Apr_{\ii}$ for each $i$.
 
To prove that $w\in\Wsp(U)$ and estimate \eqref{stima00} 
in case $E=U\times U$ we use \eqref{locdef} to get
\begin{equation}\label{stima0}
|w|_{\Wsp(U)}^p=|u|_{\Wsp(\Ac)}^p+|w|_{\Wsp(\A)}^p+2\defectsp(w,\A\times (\Ac)).
\end{equation}
In order to control the last two terms on the rhs above we will use 
two different splitting of the oscillation $w(x)-w(y)$ corresponding
roughly to short range and long range interactions estimates. 

First decompose further the seminorm of $w$ on $\A$ in \eqref{stima0} 
as follows, 
\begin{equation}\label{stima1}
|w|_{\Wsp({\A})}^p=\sum_{\ii\in\IL(U)}\int_{\A_{\ii}\times \A_{\ii}}\ldots dxdy
+\sum_{\{(\ii,\kk):\,\ii\neq\kk\}}\int_{\A_{\ii}\times \A_k}\ldots dxdy=:I_1+I_2.
\end{equation}
Next we deal with $I_1$: since $\varphi_{\lll}\equiv 0$ on 
$\A_{\ii}$ for $\lll\neq\ii$ and $0\leq\varphi_{\ii}\leq 1$ we have
\begin{eqnarray*}
\lefteqn{I_1=\sum_{\ii\in\IL(U)}\int_{\A_{\ii}\times \A_{\ii}}
\frac{\left|(1-\varphi_{\ii}(x))u(x)-(1-\varphi_{\ii}(y))u(y)
+(\varphi_{\ii}(x)-\varphi_{\ii}(y))z_{\ii}\right|^p}
{|x-y|^{n+sp}}dxdy}\\&&
\stackrel{\pm (1-\varphi_{\ii}(x))u(y)}{\leq}2^{p-1}\sum_{\ii\in\IL(U)}
\left(|u|^p_{\Wsp(\A_{\ii})}+\int_{\A_{\ii}\times \A_{\ii}}
\frac{\left|(\varphi_{\ii}(y)-\varphi_{\ii}(x))(u(y)-z_{\ii})\right|^p}
{|x-y|^{n+sp}}dxdy\right).
\end{eqnarray*}
Take note that $\Lip(\varphi_{\ii})\leq 2m^2/\rho$ for all $\ii\in\IL(U)$, 
then Fubini theorem and (i) in Lemma~\ref{Adams} applied with 
$\nu=n+(s-1)p$ and $O=\A_{\ii}$ imply
\begin{eqnarray}\label{stimaI1}
\lefteqn{I_1\leq 2^{p-1}\sum_{\ii\in\IL(U)}\left(|u|^p_{\Wsp(\A_{\ii})}+
\left(\frac{2m^2}{\rho}\right)^p\int_{\A_{\ii}\times \A_{\ii}}
\frac{|u(y)-z_{\ii}|^p}{|x-y|^{n+(s-1)p}}dxdy\right)}\notag\\&&
\leq \cnps\sum_{\ii\in\IL(U)}\left(|u|^p_{\Wsp(\A_{\ii})}+
m^{2p}\rho^{-sp}\int_{\A_{\ii}}|u(y)-z_{\ii}|^pdy\right).
\end{eqnarray}
To estimate $I_2$ we rewrite $w$ as follows
\begin{eqnarray}\label{triangolare2}
w(x)-w(y)=u(x)-u(y)+\sum_{\lll\in\IL(U)}\varphi_\lll(x)(z_\lll-u(x))+
\sum_{\lll\in\IL(U)}\varphi_\lll(y)(u(y)-z_\lll).
\end{eqnarray}
With the help of \eqref{triangolare2} and since $\varphi_\lll\equiv 0$ 
on $\A\setminus\A_\lll$ we infer
$$
I_2=\sum_{\{(\ii,\kk):\,\ii\neq\kk\}}\int_{\A_{\ii}\times \A_\kk}
\frac{|u(x)-u(y)+\varphi_{\ii}(x)(z_{\ii}-u(x))+
\varphi_\kk(y)(u(y)-z_\kk)|^p}{|x-y|^{n+sp}}dxdy.
$$ 
Thus, we can bound each summand in $I_2$ as follows 
\begin{eqnarray}\label{stima2}
\lefteqn{\int_{\A_{\ii}\times\A_\kk}
\frac{\left|w(x)-w(y)\right|^p}{|x-y|^{n+sp}}dxdy\leq 3^{p-1}
\defectsp(u,\A_{\ii}\times\A_\kk)}\\&&
+ 3^{p-1}\int_{\A_{\ii}}dx\int_{\A_\kk}
\frac{\left|\varphi_{\ii}(x)(u(x)-z_{\ii})\right|^p}{|x-y|^{n+sp}}dy
+ 3^{p-1}\int_{\A_\kk}dy\int_{\A_{\ii}}
\frac{\left|\varphi_\kk(y)(u(y)-z_\kk)\right|^p}
{|x-y|^{n+sp}}dx.\notag
\end{eqnarray}
Further, take note that for every fixed $(\ii,\kk)$, with $\ii\neq\kk$, 
if $x\in\A_\ii$ and $y\in\A_\kk$ we have 
$|x-y|\geq \dist(\A_\ii,\A_\kk)\geq|x^\ii-x^\kk|-\rL$, and then
$|x-y|\geq|x^\ii-x^\kk|/2$ since $|x^\ii-x^\kk|\geq 2\rL$. Hence, 
being $\Ln(A_\kk)\leq \omega_n\rho^n$ for all $\kk\in\Z^n$, 
from \eqref{e:counting2} and the choice $\rho\in(0,\rL/2)$ it follows
for some constant $c=c(n)>0$
\begin{eqnarray}\label{raggruppamento}
\lefteqn{\sum_{\{\kk:\,\kk\neq\ii\}}\int_{\A_\kk}\frac{1}{|x-y|^{n+sp}}dy\leq 
\cnps\sum_{\{\kk:\,\kk\neq\ii\}}\frac{\rho^n}{|x^\ii-x^\kk|^{n+sp}}}\\&&
\leq \cnps\sum_{h\geq 1}\sum_{\{\kk\neq\ii,\,\,h\rL<|x^\ii-x^\kk|_\infty\leq(h+1)\rL\}}
\frac{\rho^n}{(h\rL)^{n+sp}}\leq\frac{\cnps}{\rho^{sp}}
\sum_{h\geq 1}\frac 1{h^{1+sp}}.\notag
\end{eqnarray}
By summing up on $(\ii,\kk)$, $\ii\neq\kk$, Fubini theorem, 
\eqref{stima2} and \eqref{raggruppamento} imply for some 
$\cnps=c(n,p,s)>0$
\begin{eqnarray}\label{stimaI2}
I_2\leq 3^{p-1}\defectsp(u,\cup_{\ii}(\A_{\ii}\times\cup_{\kk\neq\ii}\A_{\kk}))
+\frac{\cnps}{\rho^{sp}}\sum_{\ii\in\IL(U)}\int_{\A_{\ii}}|u(x)-z_{\ii}|^pdx,
\end{eqnarray}
We turn now to the locality defect term in \eqref{stima0}.
We use again equality \eqref{triangolare2} and notice that 
$\varphi_{\ii}\equiv 0$ on $U\setminus\overline{\A}_{\ii}$ 
to infer
\begin{eqnarray}\label{stima3}
\lefteqn{\hskip-1.5cm\defectsp(w,\A\times {(\Ac)})
\leq 2^{p-1}\defectsp(u,\A\times {(\Ac)})
+2^{p-1}\int_{\A\times (\Ac)}\frac{\left|\sum_{\IL(U)}
\varphi_{\ii}(x)(u(x)-z_{\ii})\right|^p}{|x-y|^{n+sp}}dxdy}\notag\\&&
=2^{p-1}\defectsp(u,\A\times {(\Ac)})+
2^{p-1}\sum_{\ii\in\IL(U)}\int_{\A_{\ii}\times(\Ac)}
\frac{\left|\varphi_{\ii}(x)(u(x)-z_{\ii})\right|^p}{|x-y|^{n+sp}}dxdy.
\end{eqnarray}
We fix $\ii\in\Z^n$ and define $\Delta^\prime_{\ii}=
(\A_{\ii}\times (\Ac))\cap\Delta_\rho$, then by using 
$|\varphi_{\ii}(x)-\varphi_{\ii}(y)|\leq 2m^2|x-y|/\rho$, with 
$\varphi_{\ii}(y)=0$ for $y\in\Ac$,
Fubini theorem and a direct integration yield
\begin{eqnarray}\label{0calcolo0}
\lefteqn{\int_{\Delta^\prime_{\ii}}
\frac{|\varphi_{\ii}(x)(u(x)-z_{\ii})|^p}{|x-y|^{n+sp}}dxdy
\leq (2m^2)^p\rho^{-p}\int_{\Delta^\prime_{\ii}}
\frac{|u(x)-z_{\ii}|^p}{|x-y|^{n+(s-1)p}}dxdy}\\&&
\leq (2m^2)^p\rho^{-p}\int_{\A_{\ii}}dx\int_{B_\rho(x)}
\frac{|u(x)-z_{\ii}|^p}{|x-y|^{n+(s-1)p}}dy
=\frac{n\omega_n(2m^2)^p}{p(1-s)}\rho^{-sp}
\int_{\A_{\ii}}|u(x)-z_{\ii}|^pdx.\notag
\end{eqnarray}
Let now $\Delta^{\prime\prime}_{\ii}=(\A_{\ii}\times (\Ac))
\setminus\Delta_\rho$, then we argue as above using 
$|\varphi_{\ii}(x)-\varphi_{\ii}(y)|\leq 1$ to get again 
by a direct integration
\begin{eqnarray}\label{0calcolo1}
\lefteqn{\int_{\Delta^{\prime\prime}_{\ii}}
\frac{|\varphi_{\ii}(x)(u(x)-z_{\ii})|^p}{|x-y|^{n+sp}}dxdy
\leq \int_{\Delta^{\prime\prime}_{\ii}}
\frac{|u(x)-z_{\ii}|^p}{|x-y|^{n+sp}}dxdy}\notag\\&&
\leq\int_{\A_{\ii}}dx \int_{\Rn\setminus B_\rho(x)}
\frac{|u(x)-z_{\ii}|^p}{|x-y|^{n+sp}}dy
=\frac{n\omega_n}{sp}\rho^{-sp}\int_{\A_{\ii}}|u(x)-z_{\ii}|^pdx.
\end{eqnarray}
By taking into account \eqref{stima3}-\eqref{0calcolo1} we deduce 
\begin{eqnarray}\label{stima4}
\defectsp(w,\A\times {(\Ac)})\leq 2^{p-1}\defectsp(u,\A\times {(\Ac)})
+\cnps m^{2p}\rho^{-sp}\sum_{\ii\in\IL(U)}\int_{\A_{\ii}}|u(x)-z_{\ii}|^pdx.
\end{eqnarray}
By collecting \eqref{stima0}, \eqref{stimaI1}, \eqref{stimaI2}, 
and \eqref{stima4} we infer $w\in\Wsp(U)$ 
and estimate \eqref{stima00} for $U\times U$. 

For any measurable subset $E$ in $U\times U$, being $w=u$ 
on $\Ac$, we have 
\begin{multline*}
|\defectsp(w,E)-\defectsp(u,E)|=
|\defectsp(w,E\cap(U\times A))-\defectsp(u,E\cap(U\times A))|\\
\leq |w|^p_{\Wsp(A)}+\defectsp(w,(U\setminus A)\times A)+
\defectsp(u,U\times A).
\end{multline*}
Eventually, \eqref{stima00} follows at once by taking into account 
\eqref{stima1}-\eqref{stima4}.
\end{proof}

By scaling Lemma~\ref{tecnico2} we establish
a joining lemma for fractional type energies. 

Before starting the proof we fix some notation: 
having fixed $m\in\N$ and set $\Ieps:={\mathcal I}_{\Lambda_j}(U)$, 
for all $\ii\in \Ieps$ and $h\in\N$ let
\begin{eqnarray*}
\Bijh=\{x\in\Rn :\, |x-\xiij|<m^{-3h}\rj\},\quad  
\Cijh:=\{x\in \Rn:\, m^{-3h-2}\rj<|x-\xiij|< m^{-3h-1}\rj\}.
\end{eqnarray*} 
Clearly, we have $\Cijh\subset\Bijh\setminus\overline{B}_j^{i,h+1} \subset\Vij$.
\begin{lemma}\label{joining}
Let $(u_j)_{j\in\N}$ be converging to $u$ in $L^p(U)$ with
$\sup_j|u_j|_{\Wsp(U)}<+\infty$. 
With fixed $m$, $N\in\N$, for every $j\in\N$ there exists 
$h_j\in\{1,\ldots,N\}$ and a function $w_j\in\Wsp(U)$ such that 
\begin{equation}
  \label{e:equal1}
  w_j\equiv u_j \text{ on }  U\setminus\cup_{\ii\in \Ieps}
(\overline{B}_j^{i,h_j}\setminus B_j^{i,h_j+1}),
\end{equation}
\begin{equation}
  \label{e:equal2}
  w_j(x)\equiv (u_j)_{\Cij}\; \text{ on } \Cij,
\end{equation}
for some $\cnpsm=c(n,p,s,m)>0$ it holds for every measurable set $E$ 
in $U\times U$
\begin{equation}
  \label{e:errorest}
\left|\defectsp(u_j,E)-\defectsp(w_j,E)\right|
\leq\frac{\cnps}N|u_j|_{\Wsp(U)}^p,
\end{equation}
and the sequences $(w_j)_{j\in\N}$, $(\zeta_j)_{j\in\N}$, with 
$\zeta_j:=\sum_{\ii\in \Ieps(U)}(u_j)_{\Cij}\chi_{\Vij}$,
converge to $u$ in $L^p(U)$.

In addition, if $u_j\in L^\infty(U)$
\begin{equation}
  \label{e:equal3}
  \|w_j\|_{L^\infty(U)}\leq\|u_j\|_{L^\infty(U)}. 
\end{equation}
\end{lemma}
\begin{proof}
Given $m$, $N\in\N$, then for every $j\in\N$ and $h\in\{1,\ldots,N\}$ 
fixed, apply Lemma~\ref{tecnico2} with $(\Apr)_{\ii}^h:=\Cijh$, 
$\A_{\ii}^h:=\Bijh\setminus\overline{B}_j^{i,h+1}$,
$z_{\ii}=(u_j)_{\Cijh}$, $\ii\in \Ieps$. Take note that $\rho=m^{-3h}\rj$. 
If $w^{i,h}_j$ denotes the resulting function and 
$\A^h=\cup_{\ii\in\Ieps}\A_{\ii}^h$,
then for some constant $\cnps=c(n,p,s)$ and for any measurable set 
$E$ in $U\times U$ by \eqref{stima00} it holds
\begin{eqnarray*}
\left|\defectsp(u_j,E)-\defectsp(w_j,E)\right|\leq 
\cnps\defectsp(u_j,U\times\A^h)
+\cnps m^{2p}\left(\frac{m^{3h}}{\rj}\right)^{ps}\sum_{\ii\in\Ieps}
\int_{\A_{\ii}^h}|u_j-(u_j)_{\Cijh}|^pdx.
\end{eqnarray*}
This estimate, together with the scaled Poincar\`e-Wirtinger 
inequality \eqref{e:PWscaled} with $r=m^{-3h}\rj$, gives 
\begin{equation}\label{e:errorest1}
\left|\defectsp(u_j,E)-\defectsp(w_j,E)\right|
\leq\cnps\left(\defectsp(u_j,U\times\A^h)
+|u_j|_{\Wsp(\A^h)}^p\right)\leq \cnps
\defectsp(u_j,U\times\A^h),
\end{equation}
fro some $\cnpsm=c(n,p,s,m)>0$.
By summing up and averaging on $h$, being the $\A^h$'s
disjoint, we find $h_j\in\{1,\ldots,N\}$ such that 
\begin{equation}\label{e:errorest2}
\defectsp(u_j,U\times\A^{h_j})\leq\frac{1}N\defectsp(u_j,U\times\cup_{h}\A^h).
\end{equation}
Set $w_j:=w^{i,h_j}_j$, then \eqref{e:equal1} and \eqref{e:equal2} 
are satisfied by construction, and moreover 
\eqref{e:errorest1} and \eqref{e:errorest2} imply \eqref{e:errorest}.

To prove that $(w_j)_{j\in\N}$ converges to $u$ in $L^p(U)$ we use  
\eqref{e:PWscaled}, with $r=m^{-3h}\rj$, and the very definition of 
$w_j$ as convex combination of $u_j$ and the mean value 
$(u_j)_{C_j^{i,h_j}}$ on $\Bijhj\setminus\overline{B}_j^{i,h_j+1}$ to get
\begin{eqnarray*}
\lefteqn{\|u_j-w_j\|_{L^p(U)}^p=\|u_j-w_j\|_{L^p(\A^{h_j})}^p
=\sum_{\ii\in\Ieps}\|u_j-w_j\|_{L^p(\Bijhj\setminus B_j^{i,h_j+1})}^p}\\&&
\leq\sum_{\ii\in \Ieps}\|u_j-(u_j)_{C_j^{i,h_j}}\|
_{L^p(\Bijhj\setminus B_j^{i,h_j+1})}^p\leq c\left(\frac{\rj}{m^{3h_j}}\right)^{ps}
\sum_{\ii\in \Ieps}|u_j|_{\Wsp(\Bijhj\setminus B_j^{i,h_j+1})}^p
\leq c\rj^{ps}|u_j|_{\Wsp(U)}^p,
\end{eqnarray*}
where $\cnpsm=c(n,p,s,m)>0$.

Eventually, let us show the convergence of $(\zeta_j)_{j\in\N}$ 
to $u$ in $L^p(U)$. To this aim we prove that $(\zeta_j-u_j)_{j\in\N}$ 
is infinitesimal in $L^p(U)$.
Fix any number $M$ with $\sup_j\Rj/\rj<M<+\infty$ (see \eqref{e:Rj/rj}), 
we claim that for some constant $c=c(n,p,s,m,M,N)>0$ we have
\begin{equation}\label{e:nuova}
\sum_{\ii\in \Ieps}\|u_j-(u_j)_{\Cij}\|_{L^p(\Vij)}^p
\leq c\,\rj^{sp}|u_j|_{\Wsp(U)}^p.
\end{equation}
Given this for granted the conclusion is a straightforward consequence
of the definition of $\zeta_j$, of \eqref{e:nuova}, of \eqref{e:counting}$_3$ 
and of the equi-integrability of $(|u_j|^p)_{j\in\N}$,
i.e.
\begin{eqnarray*} 
\|\zeta_j-u_j\|_{L^p(U)}^p=\sum_{\ii\in \Ieps}\|u_j-(u_j)_{\Cij}\|_{L^p(\Vij)}^p
+\|u_j\|_{L^p(U\setminus\cup_{\Ieps}\Vij)}^p.
\end{eqnarray*}
To prove \eqref{e:nuova} we use \eqref{e:PWscaled} and the fact that 
the balls $B_{M\rj}(\xiij)$ have (uniformly) finite overlapping. 
More precisely, the inclusions $\Vij\subseteq\overline{B}_{\Rj}(\xiij)
\subseteq B_{M\rj}(\xiij)$ and \eqref{e:PWscaled} applied with $r=\rj$
give for some  $c=c(n,p,s,m,M,N)>0$
$$
\|u_j-(u_j)_{\Cij}\|_{L^p(\Vij)}^p\leq\|u_j-(u_j)_{\Cij}\|_{L^p(B_{M\rj}(\xiij))}^p
\leq c\,\rj^{ps}|u_j|_{\Wsp(B_{M\rj}(\xiij))}^p.
$$
Moreover, since $\Lambda_j$ is a Delone set and by definition 
$\rj\leq r_{\Lambda _j}$, an elementary counting argument implies
$\sup_{\ii\in\Z^n}\#\{\kk\in\Z^n:\,B_{M\rj}(\xiij)\cap B_{M\rj}(\xkkj)
\neq\emptyset\}\leq(2M+1)^n$ for all $j\in\N$.

Finally, \eqref{e:equal3} follows by construction.
\end{proof}

\subsection{Proof of the $\Gamma$-convergence}\label{s:proof}

We establish the $\Gamma$-convergence result in Theorem~\ref{main}. 
It will be a consequence of Propositions~\ref{lb}, \ref{ub}, below. 
We start with the lower bound inequality.
\begin{proposition}\label{lb}
For every $u_j\to u$ in $L^p( U)$ we have 
$$
\liminf_j\Fj(u_j)\geq\FF(u).
$$
\end{proposition}
\begin{proof}
Fix $N\in\N$, $\delta>0$, and set $m=\lfloor 1/\delta\rfloor\in\N$,
$\lfloor \cdot\rfloor$ denoting the integer part function.
Consider the sequence $(w_j)_{j\in\N}$ provided by Lemma~\ref{joining}. 
We do not highlight its dependence on $\delta$, $N$ for the sake of 
notational convenience. We remark that whatever 
the choice of $\delta$ and $N$ is, we have that $(w_j)_{j\in\N}$ 
converges to $u$ in $L^p(U)$ and that for some $c=c(n,p,s,\delta)$ it holds 
\begin{equation}\label{e:reduction}
\left(1+\frac{\cnpsm}N\right)\liminf_j\Fj(u_j)\geq\liminf_j\Fj(w_j).
\end{equation}
Furthermore, we note that for $j$ sufficiently big
$\cup_{\ii\in\Ieps}(\Vij\times\Vij)\subseteq\Delta_\delta$, and thus
\begin{eqnarray}\label{e:parziale}
\lefteqn{\hskip-1cm\liminf_j\Fj(w_j)
\geq\liminf_j\left(\int_{ U\times U\setminus\Delta_\delta}
\frac{|w_j(x)-w_j(y)|^p}{|x-y|^{n+sp}}dx\,dy+
\sum_{\ii\in \Ieps}|w_j|^p_{\Wsp(\Vij)}\right)}\notag\\
&&\hskip-1cm\geq\int_{U\times U\setminus\Delta_\delta}
\frac{|u(x)-u(y)|^p}{|x-y|^{n+sp}}dx\,dy+
\liminf_j\sum_{\ii\in \Ieps}|w_j|^p_{\Wsp(\Vij)},
\end{eqnarray}
thanks to Fatou lemma. We claim that
\begin{equation}
  \label{e:capacity}
\liminf_j\sum_{\ii\in \Ieps}|w_j|^p_{\Wsp(\Vij)}\geq
\theta\,(\csp(T)-\epsilon_\delta)\int_ U|u(x)|^p\beta(x)dx,
\end{equation}
with $\epsilon_\delta>0$ infinitesimal as $\delta\to 0^+$.
Given this for granted, by \eqref{e:reduction} inequality 
\eqref{e:parziale} rewrites as
\begin{equation}\label{e:finale}
\left(1+\frac{\cnpsm}N\right)\liminf_j\Fj(u_j)\geq
\int_{ U\times U\setminus\Delta_\delta}
\frac{|u(x)-u(y)|^p}{|x-y|^{n+sp}}dx\,dy+
\theta\,(\csp(T)-\epsilon_\delta)\int_ U|u(x)|^p\beta(x)dx.
\end{equation}
The thesis then follows by passing to the limit first as $N\to+\infty$ 
and then as $\delta\to 0^+$ in \eqref{e:finale}. 
  
To conclude we are left with proving \eqref{e:capacity}. We keep the
notation of Lemma~\ref{joining}, and further set 
$\Bij:=\{x\in\Rn:\,|x-\xiij|<m^{-(3h_j+1)}\rj\}$,
for all $\ii\in\Ieps$. Take note that $\Bij\subseteq\Vij$. We have
\begin{eqnarray}\label{e:stima1}
\lefteqn{|w_j|_{\Wsp(\Vij)}^p\geq
\inf\left\{|w|_{\Wsp(\Bij)}^p:\,w\in\Wsp(\Rn),\,w=(u_j)_{\Cij}\text{ on }\Cij,\,
\tilde{w}=0\text{ q.e. on } \Kij\right\}}\\&&
=\inf\left\{|w|_{\Wsp(\Bij)}^p:\,w\in\Wsp(\Rn),\,w=0\text{ on }\Cij,\,
\tilde{w}=(u_j)_{\Cij}\text{ q.e. on } \Kij\right\}\notag\\&&
=|(u_j)_{\Cij}|^p\csp\left(\Kij,\Bij;\frac{\rj}{m^{3h_j+2}}\right)=
\lambdaj^{n-ps}|(u_j)_{\Cij}|^p
\csp\left(T,B_{\frac{\rj}{m^{3h_j+1}\lambdaj}};\frac{\rj}{m^{3h_j+2}\lambdaj}\right).
\notag
\end{eqnarray}
The last equality is justified by an elementary translation and 
scaling argument. Thanks to \eqref{e:cap2} in Lemma~\ref{loccap},
by recalling that $h_j\in\{1,\ldots,N\}$, we get the following estimate
$$
\csp\left(T,B_{\frac{\rj}{m^{3h_j+1}\lambdaj}};\frac{\rj}{m^{3h_j+2}\lambdaj}\right)
\geq\csp(T)-\frac{c}{(m-1)^{sp}}
\CSP\left(B_1,B_{\frac{\rj}{m^{3h_j+2}\lambdaj}}\right).
$$
Hence, if $A\in\AA(U)$ is such that  $A\subset\subset U$, for 
$j$ sufficiently big we infer 
\begin{eqnarray*}
\sum_{\ii\in \Ieps}|w_j|^p_{\Wsp(\Vij)}\geq
\left(\csp(T)-\frac{c}{(m-1)^{sp}}
\CSP\left(B_1,B_{\frac{\rj}{m^{3h_j+2}\lambdaj}}\right)\right)
\int_A|\zeta_j(x)|^p\Psi_j(x)dx,
\end{eqnarray*}
where $\Psi_j(x):=\sum_{\ii\in\Ieps}\lambdaj^{n-sp}
(\Ln(\Vij))^{-1}\chi_{\Vij}(x)$ and $\zeta_j$ is defined in 
Lemma~\ref{joining}. Take note that by \eqref{e:counting}$_2$ we have  
$$
\left|\int_{A^\prime}\Psi_j(x)dx-\lambdaj^{n-sp}\#(\Ieps(A^\prime))\right|
\leq\lambdaj^{n-sp}\#(\IIeps(A^\prime))\leq
\omega_n^{-1}\L^n((\partial A^\prime)_{\Rj})
$$
for any $A^\prime\in\AA(U)$.  
Testing the inequality above for all cubes in $U$ with sides 
parallel to the coordinate axes, centers and vertices with 
rational coordinates yields $\Psi_j\to \theta\,\beta$ weak$^*$ 
$L^\infty(U)$ by \eqref{e:lambdaj} and \eqref{e:unif-distrib}.

By taking this into account, the thesis follows at once by 
the convergence of relative capacities to the global one proved 
in \eqref{e:cap1} of Lemma~\ref{loccap}, the strong convergence 
of $(\zeta_j)_{j\in\N}$ to $u$  in $L^p(U)$ established in 
Lemma~\ref{joining}, and eventually by letting $A$ increase to $U$.
\end{proof}

In the next proposition we prove that the lower bound established 
in Proposition~\ref{lb} is attained. In doing that the main 
difficulty is to show that the inequalities in \eqref{e:parziale} 
are optimal. 
Due to the insight provided by Proposition~\ref{lb} we show that the 
capacitary contribution is concentrated along the diagonal set 
$\Delta$ and is due to short range interactions.
Instead, long range interactions are responsible for the non-local 
term in the limit.
\begin{proposition}\label{ub}
For every $u\in L^p(U)$ there exists a sequence $(u_j)_{j\in\N}$ 
such that $u_j\to u$ in $L^p(U)$ and  
$$
\limsup_j\Fj(u_j)\leq\FF(u).
$$
\end{proposition}
\begin{proof}
We may assume $u\in W^{1,\infty}(U)$ by a standard density argument 
and the lower semicontinuity of $\Gamma\hbox{-}\limsup\FFepsj$.

With fixed $N\in\N$, let $\xij\in\Wsp(\Rn)$ be such that $\xij=0$ on 
$\Rn\setminus\overline{B}_N$, $\tilde{\xi}_{N}\geq 1$ $\csp$ q.e. on 
$T$ and $|\xij|^p_{\Wsp(\Rn)}\leq\CSP(T,B_N)+1/N$, and let
$\zeta\in C^{\infty}_0(B_{N})$ be any function such that 
$\zeta\equiv 1$ on $B_{\rL}$, $(\Lip\zeta)^p\leq 2$ and 
$0\leq\zeta\leq 1$.

Let $(w_j)_{j\in\N}$ be the sequence obtained from $u$ by applying 
Lemma~\ref{joining} with $m=2$. We keep the notation 
introduced there and further set 
\begin{eqnarray*}
&&\Bij:=\{x\in\Rn:\,|x-\xiij|<2^{-3h_j}\rj\},\;\;
\uij:=u_{\Cij}\quad \text{for every}\; \ii\in\Ieps,\\
&&\tBij:=B_{N\lambdaj}(\xiij)\cap U\quad \text{for every}\; \ii\in\IIeps,\; 
\IIeps:={\mathscr I}_{\Lambda_j}(U),\\
&& U_j:=U\setminus\left(\left(\cup_{\Ieps}\Bij\right)\cup
\left(\cup_{\IIeps}\tBij\right)\right).
\end{eqnarray*}
Then, recalling that $\lambdaj=\rj^{n/(n-sp)}$, define
\begin{eqnarray}
  \label{e:recovery}
  u_j(x):=\begin{cases}

w_j(x) & U_j\\

\left(1-\xij\left(\frac{x-\xiij}{\lambdaj}\right)\right)\uij
& \Bij,\,\ii\in\Ieps\\

\left(1-\zeta\left(\frac{x-\xiij}{\lambdaj}\right)\right)w_j(x)
& \tBij,\,\ii\in\IIeps.
\end{cases}
\end{eqnarray}
For the sake of notational simplicity we have not highlighted the
dependence of the sequence $(u_j)_{j\in\N}$ on the parameter $N\in\N$.
Clearly, $(u_j)_{j\in\N}$ converges strongly to $u$ in $L^p(U)$, and
moreover it satisfies the obstacle condition by construction.
The rest of the proof is devoted to show that $u_j\in\Wsp(U)$ with 
$$
\limsup_j\FFepsj(u_j)\leq\FF(u)+\epsilon_\delta+\epsilon_N,
$$
where $\epsilon_\delta\to 0^+$ as $\delta\to 0^+$ and $\epsilon_N\to 0^+$ as
$N\to+\infty$. 

A first reduction can be done by computing the energy of $u_j$ only 
on a neighborhood of the diagonal $\Delta$. Indeed, Lebesgue 
dominated convergence and the stated convergence of $(u_j)_{j\in\N}$ 
to $u$ in $L^p(U)$ imply
$$
\lim_j\defectsp(u_j,(U\times U)\setminus\Delta_\delta)=
\defectsp(u,(U\times U)\setminus\Delta_\delta).
$$
In addition, since $u_j\equiv w_j$ on $U_j$ by \eqref{e:errorest}
in Lemma~\ref{joining} we have for some constant $\cnps=c(n,p,s)$
\begin{equation}\label{e:tesi0}
\limsup_j\defectsp(u_j,(U_j\times U_j)\cap\Delta_\delta)\leq
\limsup_j\defectsp(w_j,(U\times U)\cap\Delta_\delta)
\leq\left(1+\frac{\cnps}N\right)
\defectsp(u,(U\times U)\cap\Delta_\delta)=\epsilon_\delta.
\end{equation}
The conclusion then follows provided we show that
\begin{eqnarray}
  \label{e:tesi1}
\lefteqn{\hskip-1cm\limsup_j\left(
\defectsp(u_j,(U\times(U\setminus\overline{U}_j))\cap\Delta_\delta)+
\defectsp(u_j,((U\setminus\overline{U}_j)\times U_j)\cap\Delta_\delta)\right)}
\notag\\ &&
\leq\theta\,\csp(T)\int_U|u(x)|^p\beta(x)dx+\epsilon_N+\epsilon_\delta.
\end{eqnarray}
In order to prove this we introduce the following 
splitting of the left hand side above:
\begin{eqnarray}
  \label{e:decompose}
  \lefteqn{\defectsp(u_j,(U\times(U\setminus\overline{U}_j))
\cap\Delta_\delta)\leq\sum_{\ii\in\Ieps}|u_j|^p_{\Wsp(\Bij)}
+\sum_{\{(\ii,\kk)\in\Ieps^2:\,0<|\xiij-\xkkj|<\delta\}}
\defectsp(u_j,\Bij\times \Bkj)}\notag\\&&
+2\sum_{\ii\in\Ieps}\defectsp(u_j,(\Bij\times U_j)\cap\Delta_\delta)
+\sum_{(\ii,\kk)\in\IIeps^2}\defectsp(u_j,\tBij\times 
\tBkj)\notag\\&&
+2\sum_{\ii\in\IIeps}\defectsp(u_j,(\tBij\times
U_j)\cap\Delta_\delta)
+2\sum_{(\ii,\kk)\in\Ieps\times\IIeps}\defectsp\left(u_j,(\Bij\times
\tBkj)\cap\Delta_\delta\right)=:I^1_j+\ldots+I^6_j\notag.
\end{eqnarray}
Next we estimate separately each term $I^h_j$, $h\in\{1,\ldots,6\}$.
Since the computations below are quite involved, 
we divide our argument into several steps to provide a proof 
as clear as possible. Take note that all the constants $c$ appearing 
in the rest of the proof depend only on $n$, $p$, $s$, hence this
dependence will no longer be indicated.

\medskip
{\it Step 1. Estimate of  $I^1_j$:
\begin{equation}
  \label{e:dec1}
  \limsup_jI^1_j\leq\theta\,(\csp(T)+\epsilon_N)\int_U|u(x)|^p\beta(x)dx.
\end{equation}
}
A change of variables yields
\begin{eqnarray*}
   \lefteqn{I^1_j=
\lambdaj^{n-sp}
\sum_{\ii\in\Ieps}|\uij|^p|\xij|^p_{\Wsp(\lambdaj^{-1}(\Bij-\xiij))}}
\\&&\leq\left(\CSP(T,B_N)+\frac 1N\right)\sum_{\ii\in\Ieps}\rj^n|\uij|^p=
\left(\CSP(T,B_N)+\frac 1N\right)\int_U|\zeta_j(x)|^p\Psi_j(x)dx,
\end{eqnarray*}
where $\Psi_j(x)=\sum_{\ii\in\Ieps}\lambdaj^{n-sp}
(\Ln(\Vij))^{-1}\chi_{\Vij}(x)$ and $\zeta_j$ is defined in 
Lemma~\ref{joining}.
Arguing as in Proposition~\ref{lb} and by Lemma~\ref{loccap} 
we conclude \eqref{e:dec1}.

\medskip
{\it Step 2. Estimate of  $I^2_j$:
\begin{equation}
  \label{e:dec2}
\limsup_jI^{2}_j\leq \epsilon_\delta.
\end{equation}
} 
Take note that by the very definition of $u_j$ in 
\eqref{e:recovery} for any $(x,y)\in \Bij\times \Bkj$, 
$\ii\neq\kk$ and $\ii,\kk\in\Ieps$, we get 
$$
u_j(x)-u_j(y)=\left(\uij-\ukj\right)
-\xij\left(\lambdaj^{-1}(x-\xiij)\right)\uij
+\xij\left(\lambdaj^{-1}(y-\xkkj)\right)\ukj.
$$
Hence, we can bound $I^2_j$ as follows
\begin{eqnarray*}
&& I^2_j\leq 
3^{p-1}\sum_{\ii\in\Ieps}\sum_{\{\kk\in\Ieps:\,0<|\xiij-\xkkj|<\delta\}}
\int_{\Bij\times \Bkj}\frac{|\uij-\ukj|^p}{|x-y|^{n+sp}}dxdy\\&&+
3^{p}\|u\|^p_{L^\infty(U)}\sum_{\ii\in\Ieps}
\sum_{\{\kk\in\Ieps:\,0<|\xiij-\xkkj|<\delta\}}
\int_{\Bij\times \Bkj}\frac{|\xij(\lambdaj^{-1}(x-\xiij))|^p}{|x-y|^{n+sp}}
dxdy=:I^{2,1}_j+ I^{2,2}_j.
\end{eqnarray*}
Since $|\xiij-\xkkj|/2\leq |x-y|\leq 2|\xiij-\xkkj|$ for any 
$(x,y)\in \Bij\times \Bkj$, $\ii,\kk\in\Ieps$ with $\ii\neq\kk$,  
we infer $|\uij-\ukj|\leq 2\Lip(u)|\xiij-\xkkj|$ being $u\in\Lip(U)$,
and so we deduce 
\begin{equation}\label{e:dec21b}
\int_{\Bij\times \Bkj}\frac{|\uij-\ukj|^p}{|x-y|^{n+sp}}dxdy
\leq \cnps\Lip^p(u)\frac{\rj^{2n}}{|\xiij-\xkkj|^{n+(s-1)p}}.
\end{equation}
To go on further we notice that for every fixed $\ii\in\Ieps$ we have
$$
\{\kk\in\Ieps:\,0<|\xiij-\xkkj|_\infty<\delta\}\subseteq
\cup_{h=2}^{\lfloor\delta/\rj\rfloor}
\{\kk\in\Ieps:\,h \rj\leq|\xiij-\xkkj|_\infty<(h+1)\rj\},
$$
where $\lfloor t\rfloor$ denotes the integer part of $t$.
The latter inclusion together with \eqref{e:counting}$_1$,
\eqref{e:counting2} and \eqref{e:dec21b} entail
\begin{eqnarray}\label{e:dec21} 
\lefteqn{I^{2,1}_j\leq \cnps\Lip^p(u)\sum_{\ii\in\Ieps}
\sum_{h=2}^{\lfloor\delta/\rj\rfloor}
\sum_{\{\kk\in\Ieps:\,h \rj\leq|\xiij-\xkkj|_\infty<(h+1)\rj\}} 
\frac{\rj^{n-(s-1)p}}{h^{n+(s-1)p}}
\notag}\\&&
\stackrel{\eqref{e:counting}_1,\eqref{e:counting2}}{\leq} \cnps\Lip^p(u)
\sum_{h=2}^{\lfloor\delta/\rj\rfloor}\frac{\rj^{-(s-1)p}}{h^{1+(s-1)p}}
\leq \cnps\Lip^p(u)\delta^{(1-s)p}.
\end{eqnarray}
In the last inequality we used that 
$\sum_{h=2}^Mh^{-(1+\gamma)}\leq (M^{-\gamma})/(-\gamma)$,
for any $\gamma<0$ and $M\in\N$.

To deal with $I^{2,2}_j$ we use a similar argument. 
Indeed, for every $\ii\in\Ieps$ we have
$$
\sum_{\{\kk\in\Ieps:\,\kk\neq\ii\}}\int_{\Bkj}\frac{1}{|x-y|^{n+sp}}dy\leq 
\cnps\sum_{\{\kk\in\Ieps:\,\kk\neq\ii\}}\frac{\rj^n}{|\xiij-\xkkj|^{n+sp}}
\stackrel{\eqref{e:counting2}}{\leq} 
\frac{\cnps}{\rj^{sp}}\sum_{h\geq 1}\frac 1{h^{1+sp}}.
$$
Thus, being $\xij(\lambdaj^{-1}(\cdot-\xiij))$  supported in 
$\Bij$, a change of variables yields
\begin{equation}\label{e:dec22}
I^{2,2}_j\leq \cnps\|u\|_{L^\infty(U)}^p2^{Nn}
\lambdaj^{n}\rj^{-n-sp}\|\xij\|_{L^p(B_N)}^p=
\cnps\|u\|_{L^\infty(U)}^p2^{Nn}\rj^{\frac{(sp)^2}{n-sp}}\|\xij\|_{L^p(B_N)}^p.
\end{equation}
Clearly, \eqref{e:dec21} and \eqref{e:dec22} imply \eqref{e:dec2}.

\medskip
{\it Step 3. Estimate of  $I^3_j$:
\begin{equation}
  \label{e:dec3}
  \limsup_jI^3_j\leq 
\epsilon_\delta+\epsilon_N.
\end{equation}
} 
Being $u_j\equiv w_j$ on $U_j$ and  
$\spt(\xij(\lambdaj^{-1}(\cdot-\xiij)))\subseteq\Bij$, we find
\begin{eqnarray*}
\lefteqn{I^3_j\leq \cnps \|u\|^p_{L^\infty(U)}\sum_{\ii\in\Ieps}
\defectsp(\xij(\lambdaj^{-1}(\cdot-\xiij)),\Bij\times (U\setminus\Bij))
}\\&&
+\cnps\sum_{\ii\in\Ieps}\int_{\Bij\times U_j}
\frac{|w_j(x)-\uij|^p}{|x-y|^{n+sp}}dxdy
+\cnps\sum_{\ii\in\Ieps}\int_{(\Bij\times U_j)\cap\Delta_\delta}
\frac{|w_j(y)-w_j(x)|^p}{|x-y|^{n+sp}}dxdy\\&&=:I^{3,1}_j+I^{3,2}_j+I^{3,3}_j.
\end{eqnarray*}
Note that by a change of variables the integral $I^{3,1}_j$ rewrites as
 \begin{equation}
   \label{e:dec31}
I^{3,1}_j\leq \cnps \|u\|^p_{L^\infty(U)}\lambdaj^{n-sp}\rj^{-n}
\defectsp\left(\xij,B_{\frac{\rj}{8^{h_j}\lambdaj}}\times\left(\Rn\setminus
\overline{B}_{\frac{\rj}{8^{h_j}\lambdaj}}\right)\right)=\epsilon_N,
 \end{equation}
by Remark~\ref{r:caplodef}. 
To deal with the term $I^{3,2}_j$ we first integrate out $y$ 
and then use Hardy inequality:
\begin{equation}
  \label{e:dec32}
I^{3,2}_j\leq \cnps\sum_{\ii\in\Ieps}
\int_{\Bij}\frac{|w_j(x)-\uij|^p}{\dist^{sp}(x,\partial\Bij)}dx
\leq \cnps\sum_{\ii\in\Ieps}|w_j|^p_{\Wsp(\Bij)}
\leq \cnps\defectsp(w_j,(U\times U)\cap\Delta_\delta)
\stackrel{\eqref{e:tesi0}}{=}\epsilon_\delta.
\end{equation}
Finally, for what $I^{3,3}_j$ is concerned we have
\begin{equation}\label{e:dec33}
I^{3,3}_j\leq\cnps\defectsp(w_j,(U\times U)\cap\Delta_\delta)
\stackrel{\eqref{e:tesi0}}{=}\epsilon_\delta.
\end{equation}
By collecting \eqref{e:dec31}-\eqref{e:dec33} 
we infer \eqref{e:dec3}.

\medskip
{\it Step 4. Estimate of  $I^4_j$:
\begin{equation}
  \label{e:dec4}
  \limsup_jI^4_j\leq \epsilon_\delta.
\end{equation}
} 
The very definition of $u_j$ in \eqref{e:recovery} gives
for any $(x,y)\in \tBij\times \tBkj$, $\ii,\kk\in\IIeps$
\begin{eqnarray*}
\lefteqn{u_j(x)-u_j(y)=
\left(1-\zeta\left(\lambdaj^{-1}(x-\xiij)\right)\right)
w_j(x)-\left(1-\zeta\left(\lambdaj^{-1}(y-\xkkj)\right)\right)w_j(y)
}\\&&=
\left(1-\zeta\left(\lambdaj^{-1}(x-\xiij)\right)\right)(w_j(x)-w_j(y))+
\left(\zeta\left(\lambdaj^{-1}(x-\xiij)\right)-
\zeta\left(\lambdaj^{-1}(y-\xkkj)\right)\right)w_j(y).
\end{eqnarray*}
Distinguishing the couples of the form $(\ii,\ii)$, $\ii\in\IIeps$,
from the others we bound $I^4_j$ as follows
\begin{eqnarray*}
\lefteqn{I^4_j\leq \cnps\defectsp(w_j,(U\times U)\cap\Delta_\delta)
+\cnps\|u\|^p_{L^\infty(U)}\sum_{\ii\in\IIeps}
\left|\zeta\left(\lambdaj^{-1}(\cdot-\xiij)\right)\right|
^p_{\Wsp(\tBij)}}\\&&
+\cnps\|u\|^p_{L^\infty(U)}\sum_{\{(\ii,\kk)\in\IIeps^2:\,\ii\neq\kk\}}
\int_{\tBij\times \tBkj}
\frac{|\zeta(\lambdaj^{-1}(x-\xiij))-\zeta(\lambdaj^{-1}(y-\xkkj))|^p}
{|x-y|^{n+sp}}dxdy:=I^{4,1}_j+I^{4,2}_j+I^{4,3}_j.
\end{eqnarray*}
A change of variables and \eqref{e:counting}$_2$ yield
\begin{equation}
  \label{e:dec42}
   I^{4,1}_j+I^{4,2}_j\leq \cnps\defectsp(w_j,(U\times U)\cap\Delta_\delta)+
\cnps
\L^n\left((\partial U)_{\Rj}\right)|\zeta|^p_{\Wsp(B_N)}
\stackrel{\eqref{e:tesi0}}{\leq}\epsilon_\delta+\cnps\,
\L^n\left((\partial U)_{\Rj}\right).
\end{equation}
To deal with the term $I^{4,3}_j$ first note that
$$
I^{4,3}_j\leq \cnps\|u\|^p_{L^\infty(U)}\sum_{\ii\in\IIeps}
\int_{\tBij\times(U\setminus \tBij)}
\frac{|\zeta(\lambdaj^{-1}(x-\xiij))|^p}{|x-y|^{n+sp}}dxdy,
$$
then integrate out $y$, scale back the $x$ variable, 
and finally use Hardy inequality taking into account 
that $\zeta\in C^\infty_c(B_N)$:
\begin{multline}\label{e:dec43}
I^{4,3}_j\leq \cnps\|u\|^p_{L^\infty(U)}\sum_{\ii\in\IIeps}\int_{\tBij}
\frac{|\zeta(\lambdaj^{-1}(x-\xiij))|^p}{\dist^{sp}(x,\partial
  \tBij)}dx
=\cnps\|u\|^p_{L^\infty(U)}\lambdaj^{n-sp}\sum_{\ii\in\IIeps}\int_{B_N}
\frac{|\zeta(x)|^p}{\dist^{sp}(x,\partial B_N)}dx\\
\leq\cnps\|u\|^p_{L^\infty(U)}\lambdaj^{n-sp}\#(\IIeps)|\zeta|^p_{\Wsp(B_N)}
\leq c\|u\|^p_{L^\infty(U)}\L^n\left((\partial U)_{\Rj}\right)|\zeta|^p_{\Wsp(B_N)}
\end{multline}
by \eqref{e:counting}$_2$. In conclusion, 
\eqref{e:dec42} and \eqref{e:dec43} give \eqref{e:dec4}.

\medskip
{\it Step 5. Estimate of  $I^5_j$:
\begin{equation}
  \label{e:dec5}
  \limsup_j I^5_j\leq \epsilon_\delta.
\end{equation}
} 
The computations are similar to the previous step once one notices that
for any $(x,y)\in \tBij\times U_j$,
$\ii\in\IIeps$, it holds
$$
u_j(x)-u_j(y)=
\left(1-\zeta\left(\lambdaj^{-1}(x-\xiij)\right)\right)w_j(x)-w_j(y)
=(w_j(x)-w_j(y))-\zeta\left(\lambdaj^{-1}(x-\xkkj)\right)w_j(x).
$$
Hence, we bound $I^5_j$ by the sum of two terms, the first analogous
to $I^{4,1}_j$ and the second to $I^{4,3}_j$. Thus, \eqref{e:dec5} follows.

\medskip
{\it Step 6. Estimate of  $I^6_j$:
\begin{equation}
  \label{e:dec6}
  \limsup_jI^6_j\leq \epsilon_\delta.
\end{equation}
} 
For $(x,y)\in\Bij\times\hat{B}^\kk_j$, $\ii\in\Ieps$ 
and $\kk\in\IIeps$, we write
\begin{eqnarray*}
\lefteqn{u_j(x)-u_j(y)=
\left(1-\xij\left(\lambda_j^{-1}(x-\xiij)\right)\right)(\uij-w_j(x))}\\&&
+\left(1-\xij\left(\lambda_j^{-1}(x-\xiij)\right)\right)w_j(x)
-\left(1-\zeta\left(\lambdaj^{-1}(y-\xiij)\right)\right)w_j(y).
\end{eqnarray*}
Thus, we infer
\begin{multline*}
I^6_j\leq \cnps\sum_{(\ii,\kk)\in\Ieps\times\IIeps}
\int_{\Bij\times \tBkj}
\frac{|w_j(x)-\uij|^p}{|x-y|^{n+sp}}dxdy
+\cnps\|u\|^p_{L^\infty(U)}\sum_{(\ii,\kk)\in\Ieps\times\IIeps}
\int_{\Bij\times \tBkj}
\frac{|\zeta(\lambdaj^{-1}(y-\xkkj))|^p}{|x-y|^{n+sp}}dxdy\\
+\cnps\|u\|^p_{L^\infty(U)}\sum_{(\ii,\kk)\in\Ieps\times\IIeps:\,
0<|\xiij-\xkkj|<\delta\}}\int_{\Bij\times \tBkj}
\frac{|\xij(\lambdaj^{-1}(x-\xiij))|^p}{|x-y|^{n+sp}}dxdy
:=I^{6,1}_j+I^{6,2}_j+I^{6,3}_j.
\end{multline*}
Clearly, $I^{6,1}_j$ can be estimated as $I^{3,2}_j$,
$I^{6,2}_j$ as $I^{4,3}_j$, and $I^{6,3}_j$ as $I^{2,2}_j$.
In conclusion, \eqref{e:dec6} follows.

\medskip
{\it Step 7: Conclusion.} By collecting Step 1 - Step 6 we infer
$$
\limsup_j\FF_j(u_j)\leq\FF(u)+\epsilon_\delta+\epsilon_N,
$$
with the two terms on the rhs above infinitesimal as 
$\delta\to 0^+$ and as $N\to+\infty$, 
respectively. 
\end{proof}

\subsection{Generalizations}\label{s:generalizations}
Anisotropic and homogeneous variations of the fractional semi-norm
can be treated essentially in the same way.
Consider a $\Ln$-measurable kernel 
$K:\Rn\setminus\{\underline{0}\}\to(0,+\infty)$ such that for all
$z\in\Rn\setminus\{\underline{0}\}$ and for some constant $\alpha\geq 1$
it holds 
\begin{equation}\label{e:Kprop}
K(tz)=t^{-(n+sp)}K(z)\quad t>0,\quad\quad
\alpha^{-1} |z|^{-(n+sp)}\leq K(z)\leq \alpha |z|^{-(n+sp)}. 
\end{equation}
Define $\KK:\Wsp(\Rn)\times\AA(\Rn)\to[0,+\infty)$ as
\begin{equation}\label{e:KK}
\KK(u,A):=\int_{A\times A}K(x-y)|u(x)-u(y)|^p\,dxdy,
\end{equation}
dropping the dependence on the set of integration in case $A=\Rn$.
All the relevant quantities introduced in the preceding sections have 
analogous counterparts simply replacing the kernel 
$|\cdot|^{-(n+sp)}$ with $K$. For instance, the locality defect 
associated to the energy $\KK$ is given by
$$
\mathcal{D}_{\KK}(u,E):=\int_E K(x-y)|u(x)-u(y)|^pdxdy,
$$
for any $E\subseteq U\times U$ $\L^{n\times n}$-measurable. 
We point out that in general, $\mathcal{D}_{\KK}(u,A\times B)\neq
\mathcal{D}_{\KK}(u,B\times A)$, $A$ and $B$ being measurable subsets of 
$\Rn$. Nevertheless, a splitting formula similar to \eqref{locdef} holds.

The only relevant changes are in the proof of Lemma~\ref{loccap} in 
which we exploited the invariance of the kernel $|\cdot|^{-(n+sp)}$ 
under rotations to establish \eqref{e:cap1}.
As already noticed before the statement of Lemma~\ref{loccap}, and 
as it turns out from the proof of Propositions~\ref{lb}, \ref{ub},
in the current deterministic setting it is sufficient to prove 
pointwise convergence of the relative capacities to infer the lower 
bound estimate. This is the content of the next lemma. The argument 
below does not give uniform convergence.
\begin{lemma}\label{loccap2}
Let $r>0$ and define for every $T\subset\Rn$
$$
\CSPK(T,B_r):=\inf\left\{\KK(w):\,
w\in\Wsp(\Rn),\, w=0 \text{ on }\Rn\setminus B_r,\,
\tilde{w}\geq 1\text{ q.e. on } T\right\}.
$$ 
Then 
\begin{equation}\label{e:cap12}
\lim_{r\to+\infty}\CSPK(T,B_r)=\cspk(T).
\end{equation}
Moreover, there exists a constant $\cnps=c(n,s,p,\alpha)$ such that 
for all $0<r<R$ 
\begin{equation}\label{e:cap22}
\cspk(T)-\cspk(T,B_R;r)\leq
\frac{\cnps r^{sp}}{(R-r)^{sp}}\CSPK(T,B_r),
\end{equation}
where $\cspk(T,B_R;r):=\inf\left\{\KK(w,B_R):\,
w\in\Wsp(\Rn),\, w=0 \text{ on }\Rn\setminus B_r,\,
\tilde{w}\geq 1\text{ q.e. on } T\right\}$. 

\end{lemma}
\begin{proof}
Estimate \eqref{e:cap22} can be derived as we did for \eqref{e:cap2}
thanks to \eqref{e:Kprop}$_2$.

It is clear that $\cspk(T)\leq\CSPK(T,B_r)\leq \CSPK(T,B_R)$ for all 
$0<R<r$, so that 
$$
\lim_{r\to+\infty}\CSPK(T,B_r)\geq\cspk(T).
$$
Fix $r>0$ such that $T\subseteq B_{r/2}$ and let $u\in\Wsp(\Rn)$ be 
such that $0\leq u\leq 1$ $\Ln$ a.e. on $\Rn$, $\tilde{u}\geq 1$ q.e. 
on $T$.
Given $\varphi_r$ a cut-off function between $B_{r/2}$ and $B_{3r/4}$ 
set $w_r:=\varphi_r u$.
By construction $w_r=u$ on $B_{r/2}$, $w_r\equiv 0$ on 
$\Rn\setminus B_{3r/4}$ and $\tilde{w}_r\geq 1$ q.e. on $T$. 
We claim that $w_r\in\Wsp(\Rn)$ with
\begin{eqnarray}\label{t2}
\lefteqn{\KK(w_r,B_r)\leq \KK(u,B_{r/2})}\\&&
+c(n,p,s,\alpha)
\left(\mathcal{D}_{\KK}(u,B_{r/2}\times(B_r\setminus\overline{B}_{r/2}))
+\mathcal{D}_{\KK}(u,(B_r\setminus\overline{B}_{r/2})\times B_{r/2})
+r^{-ps}\int_{B_r\setminus\overline{B}_{r/2}}|u|^pdx\right)
\notag.
\end{eqnarray}
To this aim we bound the energy of $w_r$ on $B_r$ by
\begin{eqnarray}\label{t1}
\KK(w_r,B_r)\leq\KK(u,B_{r/2})+
\mathcal{D}_{\KK}(w_r,B_{r/2}\times(B_r\setminus\overline{B}_{r/2}))+
\mathcal{D}_{\KK}(w_r,(B_r\setminus\overline{B}_{r/2})\times B_{r/2}).
\end{eqnarray}
We estimate the first of the two locality defect terms in \eqref{t1}, 
the argument for the second being analogous. 
In doing that we can follow the argument used in Lemma~\ref{tecnico2} 
for the locality defect term (see \eqref{0calcolo0}-\eqref{stima4}).
Then, by \eqref{e:Kprop}$_2$ and since $0\leq\varphi_r\leq 1$, we infer
\begin{eqnarray*}
\lefteqn{\mathcal{D}_{\KK}(w_r,B_{r/2}\times(B_r\setminus\overline{B}_{r/2}))}
\\&&\stackrel{\pm\varphi_r(x)u(y)}{\leq}
2^{p}\left(\mathcal{D}_{\KK}(u,B_{r/2}\times(B_r\setminus\overline{B}_{r/2})
+\alpha\int_{B_{r/2}\times(B_r\setminus\overline{B}_{r/2})}
\frac{|(\varphi_r(x)-\varphi_r(y))u(y)|^p}{|x-y|^{n+sp}}dxdy\right)\\&&
\leq c(n,p,s,\alpha)
\left(\mathcal{D}_{\KK}(u,B_{r/2}\times(B_r\setminus\overline{B}_{r/2})
+ r^{-ps}\int_{B_r\setminus\overline{B}_{r/2}}|u|^pdy\right).
\end{eqnarray*}
Formula \eqref{t2} then follows at once.

Furthermore, since $w_r=0$ on $B^c_{3r/4}$, \eqref{e:Kprop}$_2$ and  
\eqref{e:Adams2} in Lemma~\ref{Adams} yield 
$$
\mathcal{D}_{\KK}(w_r,B_r\times (\Rn\setminus \overline{B}_r))\leq 
c(n,p,s)\alpha\, r^{-ps}\int_{B_r}|u|^pdx,
$$
and thus we conclude $w_r\in\Wsp(\Rn)$.

Eventually, since 
$\lim_r{\mathcal D}_{\KK}(u,\Rn\times(\Rn\setminus\overline{B}_{r/2}))=
\lim_r{\mathcal D}_{\KK}(u,(\Rn\setminus\overline{B}_{r/2})\times\Rn)=0$ 
and $u\in L^p(\R^n)$, from \eqref{t2} we conclude
$$
\lim_r\CSPK(T,B_r)\leq\limsup_r\KK(w_r,B_r)\leq\KK(u).
$$
Taking the infimum on all admissible functions $u$ we conclude.
\end{proof}

With fixed $U\in\AA(\Rn)$, consider 
$\KKepsj:L^p(U)\to[0,+\infty]$ given by
\begin{eqnarray}  \label{e:fapproxK}
  \KKepsj(u)=
  \begin{cases}
\KK(u,U) & \text{ if } u\in \Wsp(U),\,
\tildeu=0\,\,  \csp \text{ q.e. on } \Kj\cap U\\
+\infty & \text{ otherwise. }
  \end{cases}
\end{eqnarray}
\begin{theorem}\label{main2}
Let  $U\in\AA(\Rn)$ be bounded and connected with Lipschitz regular boundary. 

Then the sequence 
$(\KKepsj)_{j\in\N}$ $\Gamma$-converges in the $L^p(U)$
topology to $\KKK:L^p(U)\to[0,+\infty]$ defined by
\begin{equation}\label{e:Glimit2}
\KKK(u)=\KK(u)+\theta\,\cspk(T)\int_{U}|u(x)|^p\beta(x)\,dx
\end{equation}
if $u\in \Wsp(U)$, $+\infty$ otherwise in $L^p(U)$, where
$$
\cspk(T):=\inf\left\{\KK(w):\,w\in\Wsp(\Rn),\,\tilde{w}\geq 1
\text{ q.e. on } T\right\}.
$$
\end{theorem}
\begin{proof}
The proof is the same of Theorem~\ref{main} a part from the necessary
changes to the various technical lemmas preceding Theorem~\ref{main}.
We indicate how these preliminaries must be appropriately restated.

We have shown the (pointwise) convergence of relative capacities in 
Lemma~\ref{loccap2}.
Changing the relevant quantities according to the substitution 
of the kernel $|\cdot|^{-(n+sp)}$ with $K$, Lemmas~\ref{tecnico2} 
and \ref{joining} have analogous statements since the splitting 
formula \eqref{locdef} does.

In the proof of Proposition~\ref{lb} we use the homogeneity of the 
kernel $K$ (see \eqref{e:Kprop}$_1$) for the scaling argument in
\eqref{e:stima1} leading to the analogue of formula \eqref{e:capacity}. 
This is the reason why we ask for \eqref{e:Kprop}$_1$.

Proposition~\ref{ub} needs few changes: in the definition of 
$u_j$ choose $\xij$ to be a $(1/N)$-minimizer of $\CSPK(T,B_N)$,
the estimate of $I^1_j$ follows straightforward. 
For what the terms $I^h_j$, $h\in\{2,\ldots,6\}$ are concerned 
it suffices to take note that by condition \eqref{e:Kprop}$_2$ 
the locality defect satisfies
$\alpha^{-1}\defectsp(u,E)\leq\mathcal{D}_{\KK}(u,E)\leq\alpha\defectsp(u,E)$.
Hence, we can follow exactly Steps 2-6 to conclude.
\end{proof}
\begin{remark}
Dropping the homogeneity assumption \eqref{e:Kprop}$_1$ on the kernel 
$K$ the $\Gamma$-limit might not exist. This fact had already been 
noticed in the local case (see \cite{ANB}). 
Though, in such a framework abstract integral representation 
and compactness arguments for local functional imply 
$\Gamma$-convergence up to subsequences.
\end{remark}

\section{Random Settings}\label{s:random}
In this section we extend our analysis to two different random 
settings. 
First, we deal with obstacles having random sizes and shapes 
located on points of a periodic lattice as introduced by 
Caffarelli and Mellet \cite{Caf-Mel1}, \cite{Caf-Mel2} 
(see also \cite{Foc}). In particular, we provide a self-contained 
proof of the results in  \cite{Caf-Mel2}, \cite{Foc} avoiding 
extension techniques. 
Second, we consider random homothetics copies of a given obstacle set
placed on random Delone sets of points following the approach 
by Blanc, Le Bris and Lions to define the energy of microscopic 
stochastic lattices \cite{BLBL1} and to study some variants of the 
usual stochastic homogenization theory \cite{BLBL2}.

We have not been able to work out a unified approach for the 
two frameworks described above. The main issue for this being related 
to the interplay between the weighted version of the pointwise 
ergodic theorem in Theorem~\ref{ergo-w} below and stationarity 
for random Delone sets of points (see \eqref{e:statlatt}).

In both cases we are given a probability space 
$(\Om,\mathscr{P},\mmu)$ such that 
the group $\Z^n$ acts on $\Om$ via measure-preserving transformations 
$\tau_{\ii}:\Om\to\Om$. The $\sigma$-subalgebra of $\mathscr{P}$ of 
the invariant sets of the $\tau_\ii$'s, i.e. $\Omp\subseteq\Om$ such that 
$\tau_{\ii}\Omp=\Omp$ for all $\ii\in\Z^n$, is denoted by $\mathscr{I}$.
Recall that $(\tau_{\ii})_{\ii\in\Z^n}$ is said to be \emph{ergodic} if 
$\mathscr{I}$ is trivial, i.e. $\Omp\in\mathscr{I}$ satisfies either 
$\mmu(\Omp)=0$ or $\mmu(\Omp)=1$.

In the sequel we keep the notation introduced in Section~\ref{s:determ} 
highlighting the dependence of relevant quantities on $\om$ when needed.

\subsection{Obstacles with Random sizes and shapes}\label{caf-mell}

In this subsection we deal with the case of obstacles with random
sizes and shapes located on points of a lattice.
We restrict to the standard cubic one only for the sake of simplicity 
(see Remark~\ref{r:nonstandlatt} below for extensions). 
More precisely, let $\Lambda_j=\epsj\Z^n$ with $(\epsj)_{j\in\N}$ 
a positive infinitesimal sequence, then we have $\xiij=\epsj\ii$, 
$\Vij=\epsj(\ii+[-1/2,1/2]^n)$, $r_{\Lambda_j}=\epsj/2$ and 
$R_{\Lambda_j}=\sqrt{n}\epsj/2$. In addition, 
if $\rj=\epsj/2$ then $\theta=\Ln(U)/2^n$ and $\beta\equiv 1/\Ln(U)$. 
Thus, to simplify the presentation we will use the scaling parameter 
$\epsj$ instead of $\rj$, and the more intuitive notation $\Qij$ 
for $\Vij$.

Let us now fix the assumptions on the distribution of obstacles
originally introduced by Caffarelli and Mellet \cite{Caf-Mel1}, 
\cite{Caf-Mel2}. 
For all $\om\in\Om$ and $j\in\N$ the obstacle set $\Teo\subseteq\R^n$
is given by $\Teoe:=\cup_{\ii\in\Z^n}\Teioe$, where the sets 
$\Teioe\subseteq\Qij$, satisfy the following conditions:
\begin{itemize}
\item[{\bf (O1).}] \emph{Capacitary Scaling:}
There exists 
a process $\gamma:\Z^n\times\Om\to[0,+\infty)$ such that
for all $\ii\in\Z^n$ and $\om\in\Om$
$$
\csp(\Teioe)=\epsj^n\,\gammaio.
$$
Moreover, for some $\gammab>0$ we have for all $\ii\in\Z^n$ and
$\mmu$ a.s. $\om\in\Om$
\begin{equation}\label{e:gammabdd}
\gammaio\leq\gammab.
\end{equation}

\item[{\bf (O2).}] \emph{Stationarity of the Process:}
The process $\gamma:\Z^n\times\Om\to[0,+\infty)$ is stationary w.r.to 
the family $(\tau_{\ii})_{\ii\in\Z^n}$, i.e. for all $\ii,\kk\in\Z^n$ 
and $\om\in\Om$
\begin{equation}
  \label{e:statgamma}
\gamma(\ii+\kk,\om)=\gamma(\ii,\tau_{\kk}\om).
\end{equation}

\item[{\bf (O3).}] \emph{Strong Separation:}
There exists a positive infinitesimal sequence $(\deltaj)_{j\in\N}$, 
with $\deltaj=o(\epsj)$ and $\epsj^{1+\frac n{n-sp}}=O(\deltaj)$, such 
that $\Teioe\subseteq \xiij+\deltaj(\Qij-\xiij)$ for all 
$\ii\in\Z^n$, $\om\in\Om$.
\end{itemize}
Take note that by (O2) we have
$\gamma(\ii,\om)=\gamma(\underline{0},\tau_{-\ii}\om)$, hence
the random variables $\gamma(\ii,\om)$ are identically distributed.  
The common value of their expectations is denoted by 
$\mathbb{E}[\gamma]$, i.e. 
$$
\mathbb{E}[\gamma]:=\int_\Om\gamma(\underline{0},\om)d\mmu.
$$
Moreover, $\mathbb{E}[\gamma,\mathscr{I}]$ denotes
the \emph{conditional expectation} of the process, i.e. the unique
$\mathscr{I}$-measurable function in $L^1(\Om,\mmu)$ such 
that for every set $\Omp\in\mathscr{I}$ 
$$
\int_{\Omp}\gamma(\underline{0},\om)d\mmu=
\int_{\Omp}\mathbb{E}[\gamma,\mathscr{I}](\om)d\mmu.
$$
In the sequel we analyze the asymptotics of the energies
$\KKepsj:L^p(U)\times\Om\to[0,+\infty]$ given by
\begin{eqnarray}  \label{e:fapproxKom2}
  \KKepsj(u,\om)=
  \begin{cases}
\KK(u) & \text{ if } u\in \Wsp(U),\,
\tildeu=0\,\,  \csp \text{ q.e. on } \Teoe\cap U\\
+\infty & \text{ otherwise. }
  \end{cases}
\end{eqnarray}
where $\KK$ is the functional in \eqref{e:KK}. 
\begin{theorem}\label{main4}
Let  $U\in\AA(\Rn)$ be bounded and connected with 
Lipschitz regular boundary, and assume that the kernel $K$
satisfies \eqref{e:Kprop} and is invariant under rotations.

If (O1)-(O3) hold true, then there exists a set 
$\Om^\prime\subseteq\Om$ of full probability
such that for all $\om\in\Om^\prime$ the sequence
$(\KKepsj(\cdot,\om))$ $\Gamma$-converges in the $L^p(U)$ topology 
to $\KKK:L^p(U)\times\Om\to[0,+\infty]$ defined by
\begin{equation}\label{e:Glimit4}
\KKK(u,\om)=\KK(u)+\mathbb{E}[\gamma,\mathscr{I}]
\int_U|u(x)|^p\,dx
\end{equation}
if $u\in \Wsp(U)$, $+\infty$ otherwise in $L^p(U)$.

If in addition $(\tau_{\ii})_{\ii\in\Z^n}$ is ergodic then the
$\Gamma$-limit is deterministic, i.e. 
$\mathbb{E}[\gamma,\mathscr{I}]=\mathbb{E}[\gamma]$ $\mmu$ a.s. in $\Om$.
\end{theorem}
The proof of Theorem~\ref{main4} builds upon Theorem~\ref{main2}  
and upon a weighted ergodic theorem established in \cite[Theorem 4.1]{Foc}. 
We give the proof of the latter result for the sake of convenience.
\begin{theorem}\label{ergo-w}
Let $\gamma$ be satisfying (O2), then for every bounded open set 
$V\subset\R^n$ with $\L^n(\partial V)=0$ we have $\mmu$ a.s. 
in $\Om$
\begin{equation}
  \label{e:Birkhoff0}
\lim_j\frac{1}{\#\Ieps(V)}\sum_{\ii\in\Ieps(V)}\gammaio=
\mathbb{E}[\gamma,\mathscr{I}],
\end{equation}
and
\begin{equation}
  \label{e:Birkhoff1}
\Psi_j(x,\om):=\sum_{\ii\in\Ieps(V)}\gammaio\chi_{\Qij}(x)
\to\mathbb{E}[\gamma,\mathscr{I}]\quad  \mathrm{weak}^\ast\, L^\infty(V).
\end{equation}
\end{theorem}
\begin{proof}
Define the operators $T_{\ii}:L^1(\Om,\mmu)\to L^1(\Om,\mmu)$ by 
$T_{\ii}(f):=f\circ\tau_{\ii}$ for every $\ii\in\Z^n$.
The group property of $(\tau_{\ii})_{\ii\in\Z^n}$ implies that
$\mathscr{S}=\left\{T_{\ii}\right\}_{\ii\in\Z^n}$ is a multiparameter 
semigroup generated by the commuting isometries 
$T_{{\tt e}_r}$ for $r\in\{1,\ldots,n\}$, being
$\{{\tt e}_1,\ldots,{\tt e}_n\}$ the canonical basis of $\Rn$.

We define a process $F$ on bounded Borel sets $V$ of $\Rn$ 
with values in  $L^\infty(\Om,\mmu)$. Set $Q_1:=[-1/2,1/2]^n$ and 
let $F$ be as follows 
$$
F_V(\om):=\sum_{\{\ii\in\Z^n:\,\ii+Q_1\subseteq V\}}\gamma(\ii,\om),
$$
with the convention that $F_V(\om):=0$ if the set of summation
is empty.
It is clear that $F$ is \emph{additive}, that is it satisfies
\begin{itemize}
\item[(i)] $F$ is stationary: $T_{\ii}\circ F_{V} =F_{V+\ii}$ 
for all $\ii\in\Z^n$;

\item[(ii)] $F_{V_1\cup V_2}=F_{V_1}+F_{V_2}$, for disjoint $V_1$, $V_2$;

\item[(iii)] the random variables $F_V$ are integrable; and

\item[(iv)] the spatial constant of the process 
$\overline{\gamma}(F):=\inf\{j^{-n}\int_\Om F_{jQ_1}d\mmu\}\in(0,\infty)$.

\end{itemize}
Indeed, (i) and (iv) follow by (O2) (actually 
$\overline{\gamma}(F)=\mathbb{E}[\gamma]$), (ii) by the 
very definition of $F$, and (iii) by the positivity of 
$\gamma$ and \eqref{e:gammabdd}.
It then follows from \cite[Theorem 2]{KrPy} and 
\cite[Remark (b) p.294]{KrPy} that there exists 
$\overline{f}\in L^1(\Om,\mmu)$ such that for all 
$V\in\BB(\Rn)$ bounded with $\Ln(\partial V)=0$ and $\mmu$ a.s. 
in $\Om$
\begin{equation}\label{e:krpy}
\lim_{j\to+\infty}
j^{-n}F_{jV}(\om)=\L^n(V)\overline{f}(\om).
\end{equation}
Stationarity and boundedness of $\gamma$ together with 
\eqref{e:krpy} yield that for any $\Omp\in\mathscr{I}$ we have
$$
\int_{\Omp}\overline{f}(\om)\,d\mmu
=\lim_{j\to+\infty}j^{-n}\int_{\Omp}F_{jQ_1}(\om)\,d\mmu
=\int_{\Omp}\gamma(\underline{0},\om)\,d\mmu,
$$
so that the limit $\overline{f}$ is actually given 
by $\mathbb{E}[\gamma,\mathscr{I}]$ (see for instance 
\cite[Theorem 2.3 page 203]{Kr}).

In particular, \eqref{e:krpy} still holds by substituting $(j)_{j\in\N}$ 
with any diverging sequence $(a_j)_{j\in\N}\subseteq\N$.
Take $a_j=\lfloor 1/\epsj\rfloor$ ($\lfloor t\rfloor$ stands for the 
integer part of $t$), then notice that for every $\delta>0$ and $j$ 
sufficiently big we have
$$
a_j^{-n}\left|\sum_{\ii\in\Ieps(V)}\gamma(\ii,\om)
-F_{a_jV}(\om)\right|
\leq\gammab a_j^{-n}\#\left(\Ieps(V)\triangle
\{\ii\in\Z^n:\,a_j^{-1}(\ii+Q_1)\subseteq V\}\right)\leq
\gammab\Ln((\partial V)_\delta),
$$
here $\triangle$ denotes the symmetric difference between the 
relevant sets.
Since $\epsj^n\#\Ieps(V)\to\L^n(V)$ and $a_j\epsj\to 1$ 
we infer \eqref{e:Birkhoff0}.

Eventually, in order to prove \eqref{e:Birkhoff1}
consider the family $\mathscr{Q}$ of all open cubes in $\R^n$
with sides parallel to the coordinate axes, and with center 
and vertices having rational coordinates.
To show the claimed weak$^\ast$ convergence it suffices to check that
 $\mmu$ a.s. in $\Om$ it holds
$$
\lim_j\int_\Om\Psi_j(x,\om)\chi_Q(x)\,d\L^n=
\L^n(Q)\mathbb{E}[\gamma,\mathscr{I}]
$$
for any $Q\in\mathscr{Q}$ with $Q\subseteq V$, $V$ as in the 
statement above. We have
\begin{eqnarray*}
\left|\int_{Q}(\Psi_j(x,\om)-
\mathbb{E}[\gamma,\mathscr{I}])d\L^n\right|
\leq\left|\epsj^n\sum_{\ii\in\Ieps(Q)}\gammaio
-\L^n(Q)\mathbb{E}[\gamma,\mathscr{I}]\right|+2\gammab
\L^n\left(Q\setminus\cup_{\ii\in\Ieps(Q)}\Qij\right),
\end{eqnarray*}
and thus \eqref{e:Birkhoff0} and the denumerability of $\mathscr{Q}$
yield that the rhs above is infinitesimal $\mmu$ a.s. in $\Om$.

\end{proof}
\begin{proof}[Proof (of Theorem~\ref{main3})]
We define $\Om^\prime$ 
as any subset of full probability in $\Om$ for which Theorem~\ref{ergo-w}
holds true for $U$. 
Then, we follow the lines of the proof of Theorem~\ref{main2} pointing 
out only the necessary changes.

First, we note that by assumption (O3), in particular condition 
$\deltaj=o(\epsj)$, we can still apply Lemma~\ref{joining} in this
framework.
Hence, to get the lower bound inequality we argue as in 
Proposition~\ref{lb}, we need only to substitute \eqref{e:capacity} 
suitably. Formula \eqref{e:stima1} is replaced by 
\begin{multline*}
|w_j|_{\Wsp(\Vij)}^p\geq\epsj^n|(u_j)_{\Cij}|^p
\csp\left(\lambdaj^{-1}(\Teioe-\xiij),B_{\frac{\epsj}{m^{3h_j+1}\lambdaj}};
\frac{\epsj}{m^{3h_j+2}\lambdaj}\right)
\\\geq\left(\gammaio-\epsilon_m\right)\,\epsj^n\,|(u_j)_{\Cij}|^p,
\end{multline*}
where $\epsilon_m>0$ is infinitesimal as $m\to+\infty$.
To infer the last inequality we have taken into account (O1) and the uniform 
convergence of the relative capacities established in \eqref{e:cap2} 
of Lemma~\ref{loccap}. This is guaranteed by condition 
$\epsj^{1+\frac n{n-sp}}=O(\deltaj)$ in (O3), which ensures that the 
rescaled obstacle sets $\lambdaj^{-1}(\Teioe-\xiij)$ are equi-bounded.

By summing up all the contributions, by Theorem~\ref{ergo-w} and by 
recalling that the sequence $(\zeta_j)_{j\in\N}$ defined in Lemma~\ref{joining} 
converges strongly to $u$ in $L^p(U)$ we infer 
$$
\liminf_j\sum_{\ii\in \Ieps}|w_j|^p_{\Wsp(\Vij)}\geq
\liminf_j\int_{A}\Psi_j(x,\om)|\zeta_j|^pdx\geq
(\mathbb{E}[\gamma,\mathscr{I}]-\epsilon_m)\int_{A}|u(x)|^pdx,
$$
for all $A\in\AA(U)$ with $A\subset\subset U$. 
By increasing $A$ to $U$ and by letting $m\to+\infty$ we conclude.

The upper bound inequality is established as in Proposition~\ref{ub} 
by substituting $\xij$ in the definition of $u_j$ with $\xi_j^{\ii,N}$, 
a $(1/N)$-minimizer of $\CSPK(\lambdaj^{-1}(\Teioe-\xiij),B_N)$. 
The uniform convergence of those relative capacities is guaranteed by  
the analogue of \eqref{e:cap1} in Lemma~\ref{loccap}.
This is the reason why we suppose $K$ to be invariant under rotations.

The estimate of $I^1_j$ then follows straightforward. 
For what the terms $I^h_j$, $h\in\{2,\ldots,6\}$ are concerned,
take note that since 
$\alpha^{-1}\defectsp(u,E)\leq\mathcal{D}_{\KK}(u,E)\leq\alpha\defectsp(u,E)$
(see \eqref{e:Kprop}$_2$), we can follow exactly Steps 2-6 to conclude.
\end{proof}

\begin{remark}\label{r:nonstandlatt}
If $\Lambda$ is a generic periodic $n$-dimensional lattice 
in $\Rn$ we can argue analogously. By definition 
the points of $\Lambda$ belong to the orbit of a $\Z$-module 
generated by $n$ linearly independent vectors, 
and the Vorono\"i cells turn out to be congruent polytopes 
(see \cite[Chapter 2]{Sen}). 
Clearly, the obstacle set $\Teoe$ and stationarity for the 
process $\gamma(\cdot,\om)$ have to be defined according to the group of 
translations associated to vectors in the mentioned $\Z$-module 
which leave $\Lambda$ invariant. 
\end{remark}

\subsection{Random Delone set of points}

According to Blanc, Le Bris and Lions (see \cite{BLBL1}, 
\cite{BLBL2}) a \emph{random Delone set} is a random variable 
$\Lambda:\Om\to(\R^n)^{\Z^n}$ satisfying Definition~\ref{voro}
$\mmu$ a.s. in $\Om$. Then, both $\rL$ and $\RL:\Om\to\R$ turn out to be
random variables: $\rL$ because of its very definition 
\eqref{d:rLRL}$_1$, and $\RL$ since it can be characterized as
$\RL(\cdot)=\sup_{{\bf Q}^n}\dist(x,\Lambda(\cdot))$.

We will deal with sequences $\Lambda_j:\Om\to(\R^n)^{\Z^n}$ 
of random Delone sets fulfilling conditions analogous to 
\eqref{e:Rj/rj}-\eqref{e:unif-distrib} (see below for 
relevant examples), that is for $\mmu$ a.e. $\om\in\Om$ it holds 
\begin{equation}\label{e:Rj/rjs}
\lim_j\|\rj\|_{L^\infty(\Om,\mmu)}=0,\quad 
(1\leq)\limsup_j\|\Rj/\rj\|_{L^\infty(\Om,\mmu)}<+\infty,
\end{equation}
\begin{equation}\label{e:lambdajs}
\lim_j\#(\Lambda_j(\om)\cap U)\rj^n(\om)=\theta(\om)\in(0,+\infty), 
\end{equation}
\begin{equation}\label{e:unif-distribs}
\mu_j(\cdot,\om):=\frac 1{\#(\Lambda_j(\om)\cap U)}
\sum_{\ii\in\Lambda_j(\om)\cap U}\delta_{\xiij(\om)}(\cdot)
\rightarrow\mu(\cdot,\om):=\beta(\cdot,\om)\Ln\res U\quad w^*\hbox{-}C_b(U),
\end{equation}
for some $\BB(U)\otimes\mathscr{P}$ measurable function 
$\beta$ such that $\beta(\cdot,\om)\in L^1(U,[0,+\infty])$ 
and $\|\beta(\cdot,\om)\|_{L^1(U)}=1$. 

\begin{remark}\label{r:indices}
In \eqref{e:unif-distribs} the set of indices $\Lambda_j(\om)\cap U$ 
susbsitutes $\Ieps(U,\om)$ to ensure the measurability of $\theta$,
$\beta$. For, the measurability of $\#(\Lambda(\om)\cap U)$ follows 
easily from that of the random Delone set $\Lambda$. 

Despite this, since 
$0\leq\#(\Lambda_j(\om)\cap U)-\#\Ieps(U,\om)\leq\#\IIeps(U,\om)$, 
by \eqref{e:Rj/rjs} we have for $\mmu$ a.e. $\om\in\Om$
\begin{equation}\label{e:unif-distribs2}
\mu_j(\cdot,\om)-\frac 1{\#\Ieps(U,\om)}
\sum_{\ii\in\Ieps(U,\om)}\delta_{\xiij(\om)}(\cdot)\to 0
\quad w^*\hbox{-}C_b(U),\quad 
\lim_j\#\Ieps(U,\om)\rj^n(\om)=\theta(\om).
\end{equation}
\end{remark}
\begin{remark}
Contrary to the deterministic setting we do not know  
whether conditions \eqref{e:lambdajs}, \eqref{e:unif-distribs} 
are satisfied up to the subsequences or not. Nevertheless, 
in Examples~\ref{ex:diffeostoch1},  \ref{ex:diffeostoch2} we show 
some sets of points satisfying all the conditions listed above.
The only piece of information we extract from 
\eqref{e:Rj/rjs} is that for a subsequence (not relabeled 
in what follows) we have the convergence
\begin{equation}\label{e:w*cpt}
\langle\mu_j(\cdot,\om),\varphi\rangle_{C_b(U), {\mathcal P}(U)}\to
\langle\mu(\cdot,\om),\varphi\rangle_{C_b(U), {\mathcal P}(U)}
\quad\quad w^*\hbox{-}L^\infty(\Om,\mmu)\quad \text{ for all } 
\varphi\in C_b(U).
\end{equation}
Thus, condition \eqref{e:unif-distribs} is equivalent 
to the strong convergence in $L^1(\Om,\mmu)$ for all $\varphi\in C_b(U)$
of the (sub)sequences 
$(\langle\mu_j(\cdot,\om),\varphi\rangle_{C_b(U), {\mathcal P}(U)})$.

To establish \eqref{e:w*cpt} take note that $\mu_j$ is a 
Young Measure (see \cite{Bal}). More precisely,  
${\mu}_j:(\Om,\mathscr{P},\mmu)\to {\mathcal P}(U)$, being 
${\mathcal P}(U)$ the space of probability measures on $U$, 
is a measurable map, i.e. ${\mu}_j(A,\om)$ 
is measurable for all $A\in\BB(U)$.
By taking into account the uniform conditon \eqref{e:Rj/rjs} and 
by arguing as in Remark~\ref{r:tight} we infer that the family 
$(\mu_j)_{j\in\N}$ is parametrized tight (see \cite[Definition 3.3]{Bal}):
given $\delta>0$ for a compact set 
$C_\delta\subset U$ it holds 
$$
(\mmu\otimes{\mu}_j)(\Om\times C_\delta)=
\int_{\Om}{\mu}_j(C_\delta,\om)d\mmu(\om)\geq 1-\delta.
$$
Thus, by parametrized Prohorov theorem (see \cite[Theorem 4.8]{Bal}) 
\eqref{e:w*cpt} holds true up to a 
subsequence. 

Finally, in view of Remark~\ref{r:tight} and \eqref{e:Rj/rjs} we note that 
it is not restrictive to suppose that the limit measure in 
\eqref{e:unif-distribs} is absolutely continuous w.r.to $\L^n\res U$
$\mmu$ a.s. in $\Om$.
\end{remark}

We fix a bounded set $T\subset\R^n$ and define the obstacle 
set $\Kj(\om):=\cup_{\ii\in\Z^n}\Kij(\om)$ as the union of random 
rescaled and translated copies of $T$ according to \eqref{e:Kij}, i.e.
\begin{equation}\label{e:Kijo}
\Kij(\om):=\xiij(\om)+\lambdaj(\om) T,\quad \text{ where } 
\lambdaj(\om):=\rj(\om)^{n/(n-sp)}.
\end{equation}
The asymptotics of the energies
$\KKepsj:L^p(U)\times\Om\to[0,+\infty]$ given by
\begin{eqnarray}  \label{e:fapproxKom}
  \KKepsj(u,\om)=
  \begin{cases}
\KK(u) & \text{ if } u\in \Wsp(U),\,
\tildeu=0\,\,  \csp \text{ q.e. on } \Teoe\cap U\\
+\infty & \text{ otherwise. }
  \end{cases}
\end{eqnarray}
where $\KK$ is the functional in \eqref{e:KK}, is a
 straightforward generalization of Theorem~\ref{main}. 
\begin{theorem}\label{main3}
Let $U\in\AA(\Rn)$ be bounded and connected with Lipschitz 
regular boundary.

Assume $\Lambda_j:\Om\to(\R^n)^{\Z^n}$ is a sequence of random Delone 
sets satisfying \eqref{e:Rj/rjs}-\eqref{e:unif-distribs} $\mmu$ a.s. 
in $\Om$. Then, $\mmu$ a.s. in $\Om$ the sequence
$(\KKepsj(\cdot,\om))$ $\Gamma$-converges in the $L^p(U)$ topology 
to $\KKK:L^p(U)\times\Om\to[0,+\infty]$ defined by
\begin{equation}\label{e:Glimit3}
\KKK(u,\om)=\KK(u)+\theta(\om)\,\cspk(T)\int_U|u(x)|^p\beta(x,\om)\,dx
\end{equation}
if $u\in \Wsp(U)$, $+\infty$ otherwise in $L^p(U)$.
\end{theorem}
To ensure that the $\Gamma$-limit is deterministic we introduce 
a sequential version of stationarity for random lattices as 
defined by Blanc, Le Bris and Lions~\cite{BLBL1}, \cite{BLBL2}: 
there exists a positive and infinitesimal sequence $(\deltaj)_{j\in\N}$ 
such that for all $\ii\in\Z^n$ and $\mmu$ a.s. in $\Om$.
\begin{equation}\label{e:statlatt}
\Lambda_j(\tau_{\ii}\om)=\Lambda_j(\om)-\ii\deltaj.
\end{equation}
\begin{corollary}\label{c:rndmdelonedet}
If the assumptions of Theorem~\ref{main3} hold true, if
$(\Lambda_j)_{j\in\N}$ satisfies \eqref{e:statlatt}, and if
the family $(\tau_{\ii})_{\ii\in\Z^n}$ is ergodic, then 
$r_{\Lambda_j}$ is constant $\mmu$ a.s. in $\Om$ for all $j\in\N$, and 
there exists $\hat{\beta}\in L^1(U)$ such that for $\mmu$ 
a.e. $\om\in\Om$ we have
$\beta(\cdot,\om)=\hat{\beta}$ $\L^n$ a.e. in $U$. 

In addition, if $\rj$ is constant $\mmu$ a.s. in $\Om$ for all $j\in\N$,
then $\theta$ is constant $\mmu$ a.s. in $\Om$.
In particular, the $\Gamma$-limit in \eqref{e:Glimit3} is deterministic,
i.e. $\mmu$ a.s. in $\Om$  the functional in \eqref{e:Glimit3} takes the form
$$
\KKK(u,\om)=\KK(u)+\theta\,\cspk(T)\int_U|u(x)|^p\hat{\beta}(x)\,dx
$$
for every $u\in \Wsp(U)$, $+\infty$ otherwise in $L^p(U)$.
\end{corollary}
\begin{proof}
We show that $r_{\Lambda_j}$
is invariant under the action of 
$(\tau_{\ii})_{\ii\in\Z^n}$, this suffices to conclude since ergodicity 
implies that the only invariant random variables are $\mmu$ a.s. equal 
to constants. 
We work with $\Ieps(A,\om)$ in place of $\Lambda_j(\om)\cap A$, since
$0\leq\#(\Lambda_j(\om)\cap A)-\#\Ieps(A,\om)\leq\#\IIeps(A,\om)$ 
for all $A\in\AA(U)$ as already pointed 
out for $A=U$ in Remark~\ref{r:indices}. 

Given any $\ii\in\Z^n$, \eqref{e:statlatt} 
yields $V_j^{\kk}(\tau_{\ii}\om)=V_j^{\kk_{\ii}}(\om)-\ii\deltaj$ for some 
$\kk_{\ii}\in\Z^n$, which implies in turn that 
$r_{\Lambda_j}(\tau_\ii\om)=r_{\Lambda_j}(\om)$, and thus
$r_{\Lambda_j}$ is equal to a constant $\mmu$ a.s. in $\Om$
for every $j\in\N$. 

In addition, we have also that 
$\#\Ieps(A,\tau_{\ii}\om)=\#\Ieps(A+\ii\deltaj,\om)$ 
for every $A\in\AA(U)$, then if $\L^n(\partial A)=0$ for any $\delta>0$ 
and $j$ sufficiently big it holds $\#\Ieps(A_{-\delta},\om)\leq
\#\Ieps(A,\tau_\ii\om)\leq\#\Ieps(A_{\delta},\om)$ (see \eqref{e:vdelta} 
for the definition of $A_{\delta}$ and $A_{-\delta}$). In particular, 
we infer 
\begin{equation}\label{e:stazerror}
|\#\Ieps(A,\tau_\ii\om)-\#\Ieps(A,\om)|\leq
\#(\Ieps\cup\IIeps)(A_\delta\setminus\overline{A},\om)\vee
\#(\Ieps\cup\IIeps)(A\setminus\overline{A_{-\delta}},\om).
\end{equation}
If $\om\in\Om$ is such that \eqref{e:unif-distribs} holds,
for any $x_0\in U$ and $r\in(0,\dist(x_0,\partial U))$ we have 
by \eqref{e:counting}, \eqref{e:unif-distribs2}$_1$ and 
\eqref{e:stazerror}
$$
\int_{B_r(x_0)}\beta(x,\tau_{\ii}\om)dx
=
\lim_j\frac{\#\Ieps(B_r(x_0),\tau_{\ii}\om)}{\#\Ieps(U,\tau_{\ii}\om)}
=
\lim_j\frac{\#\Ieps(B_r(x_0),\om)}{\#\Ieps(U,\om)}
=
\int_{B_r(x_0)}\beta(x,\om)dx.
$$
In turn, from the latter equality we infer that if 
$$
\hat{\beta}(x_0,\om):=\limsup_{r\to 0^+}\fint_{B_r(x_0)}\beta(x,\om)dx,
$$
then $\hat{\beta}(x_0,\om)=\hat{\beta}(x_0,\tau_{\ii}\om)$ for all
$\ii\in\Z^n$, $x_0\in U$ and for $\mmu$ a.e. $\om\in\Om$. Thus, 
$\hat{\beta}(x_0,\om)$ is  $\mmu$ a.s. equal to a constant for every
$x_0\in U$; Lebesgue-Besicovitch differentiation theorem yields the 
conclusion.

Eventually, suppose that $\rj$ is constant, then \eqref{e:stazerror}, 
$\L^n(\partial U)=0$ and \eqref{e:Rj/rjs} imply 
that $\theta(\tau_\ii\om)=\theta(\om)$ for all $\ii\in\Z^n$. 
In conclusion, $\theta$ is equal to a constant $\mmu$ a.s. in $\Om$. 
\end{proof}

We discuss some examples related to sets of points introduced 
in \cite{BLBL1} and \cite{BLBL2}. 
As before, $(\epsj)_{j\in\N}$ denotes a positive 
infinitesimal sequence.

\begin{example}
Let us consider ensembles of points which are stationary perturbations 
of a standard periodic lattice. More precisely, given a random 
variable $X:(\Om,\mathscr{P},\mmu)\to\Rn$ define the family 
$(X_\ii)_{\ii\in\Z^n}$ by $X_\ii(\om):=X(\tau_\ii(\om))$.
By construction $(X_\ii)_{\ii\in\Z^n}$ is stationary w.r.to 
$(\tau_{\ii})_{\ii\in\Z^n}$. Let $\Lambda(\om):=\{\xii(\om)\}_{\ii\in\Z^n}$ 
where $\xii(\om):=\ii+X_\ii(\om)$, and $\Lambda_j(\om):=\epsj\Lambda(\om)$. 
A simple computation shows that \eqref{e:statlatt} is satisfied
with $\deltaj=\epsj$. 
Moreover, if $M:=\sup_{\ii}\|X_{\ii}-X\|_{L^\infty(\Om,\mmu)}<1$,
then $\Lambda$ is a random Delone set with
$1-M\leq r_{\Lambda}\leq R_{\Lambda}\leq 1+M$ $\mmu$ a.s. in $\Om$. 
In conclusion, $(\Lambda_j)_{j\in\N}$ satisfies both \eqref{e:Rj/rjs} 
and \eqref{e:statlatt}. We do not know which additional conditions 
must be imposed on $X$ to ensure \eqref{e:lambdajs} 
and \eqref{e:unif-distribs}.
\end{example}

\begin{example}\label{ex:diffeostoch1}
We consider a stochastic diffeomorphism as introduced by Blanc, 
Le Bris and Lions \cite{BLBL1}; that is a field 
$\Phi:\Rn\times\Om\to\Rn$ such that $\Phi(\cdot,\om)$ is a diffeomorphism 
for $\mmu$ a.e. $\om\in\Om$ satisfying
$$
\mathrm{ess\hbox{-}sup}_{\Rn\times\Om}\|\nabla\Phi(x,\om)\|\leq M<+\infty,
\quad \text{and }\quad 
\mathrm{ess\hbox{-}inf}_{\Rn\times\Om}\det\nabla\Phi(x,\om)\geq\nu>0.
$$ 
Recalling Example~\ref{ex:diffeo1} (see also \cite[Proposition 4.7]{BLBL1}),
let $\Lambda_j(\om)=\{\Phi(\epsj\ii,\om)\}_{\ii\in\Z^n}$,
then  $\mmu$ a.s. in $\Om$ it holds 
$(\nu M^{1-n}/2)\epsj\leq r_{\Lambda_j}(\om)\leq R_{\Lambda_j}(\om)\leq M\epsj$ 
and
$$
\beta(x,\om)=\left(\int_U\det\nabla\Phi^{-1}(x,\om)\,dx\right)^{-1}
\det\nabla\Phi^{-1}(x,\om).
$$ 
By choosing $\rj$ such that $\rj/\epsj\to\gamma>0$ $\mmu$ a.s. 
in $\Om$, we have $\theta=\gamma^n\int_U\det\nabla\Phi^{-1}(x,\cdot)\,dx$
$\mmu$ a.s. in $\Om$.
In conclusion, $\Lambda_j$ are random Delone sets satisfying 
\eqref{e:Rj/rjs}-\eqref{e:unif-distribs}.
\end{example}
\begin{example}\label{ex:diffeostoch2}
Consider as before a stochastic diffeomorphism,
and assume in addition $\Phi$ to be stationary w.r.to 
$(\tau_{\ii})_{\ii\in\Z^n}$, that is for every $\ii\in\Z^n$,
for a.e. $x\in\Rn$ and for $\mmu$ a.s. in $\Om$ it holds 
$$
\Phi(x,\tau_{\ii}(\om))=\Phi(x+\ii,\om).
$$ 
Let $\Lambda_j(\om)=\{\epsj\Phi(\ii,\om)\}_{\ii\in\Z^n}$, 
we have $(\nu M^{1-n}/2)\epsj\leq r_{\Lambda_j}(\om)\leq 
R_{\Lambda_j}(\om)\leq M\epsj$. Thus $\Lambda_j$ are random 
Delone sets satisfying \eqref{e:Rj/rjs}.
For what \eqref{e:lambdajs} and \eqref{e:unif-distribs} are 
concerned by \cite[Lemmas 2.1, 2.2]{BLBL0} and \cite[Remark 1.9]{BLBL1}, 
$\mmu$ a.s. in $\Om$ it holds
$$
\#(\Lambda_j(\om)\cap U)\epsj^n\to\left(\det\left(\mathbb{E}
\left[\int_{[0,1]^n}\nabla\Phi(x,\cdot)\,dx\right]\right)\right)^{-1}
\L^n(U).
$$ 
Then by using \eqref{e:unif-distrib2} it follows
$$
\frac 1{\#(\Lambda_j(\om)\cap U)}\sum_{\ii\in\Lambda_j(\om)\cap U}
\delta_{\xiij(\om)}\rightarrow\frac{1}{\L^n(U)} 
\Ln\res U\quad w^*\hbox{-}C_b(U).
$$
Take also note that $\Lambda_j$ satisifes \eqref{e:statlatt} 
with $\deltaj\equiv 0$. Hence, $r_{\Lambda_j}$ is constant
by Corollary~\ref{c:rndmdelonedet}, and moreover if we choose 
$\rj\in(0,\nu M^{1-n}\epsj/2)$ such that $\rj/\epsj\to \gamma>0$, 
the $\Gamma$-limit of the energies in \eqref{e:fapproxKom} is 
given by the functional 
$$
\KKK(u)=\KK(u)+\frac {\gamma^n\,\cspk(T)}{\det\left(\mathbb{E}\left[
\int_{[0,1]^n}\nabla\Phi(x,\cdot)\,dx\right]\right)}
\int_U|u(x)|^p\,dx.
$$
for $u\in\Wsp(U)$, $\KKK(u)\equiv+\infty$ otherwise in $L^p(U)$.

Eventually, let us remark that stationary diffeomorphisms according 
to Blanc, Le Bris and Lions \cite{BLBL1} satisfy the weaker condition
$\nabla\Phi(x,\tau_{\ii}\om)=\nabla\Phi(x+\ii,\om)$.  
The sets of points generated by $\Phi$ are not stationary 
according to \eqref{e:statlatt} in general.
\end{example}


\appendix

\section{ }\label{conti}

We prove some elementary bounds on the singular kernels that were  
crucial in the computations of subsections~\ref{s:technical} 
and \ref{s:proof}.
\begin{lemma}\label{Adams}
Let $\nu>0$, then there exists a positive constant 
$c(n,\nu)$ such that for every measurable set $O\subset\Rn$ 
it holds
\begin{itemize}
\item[(i)] if $\nu\in(0,n)$ and $\dist(z,O)=0$ 
 \begin{equation}
    \label{e:Adams1}
\int_{O}\frac{1}{|x-z|^{\nu}}dx\leq c(n,\nu)(\Ln(O))^{1-\nu/n}.
\end{equation}
\item[(ii)]  if $\nu\in(n,+\infty)$ and $\dist(z,O)>0$ 
  \begin{equation}
    \label{e:Adams2}
\int_{O}\frac 1{|x-z|^{\nu}}\dx\leq
c(n,\nu)\left(\dist(z,O)\right)^{n-\nu},
  \end{equation}
\end{itemize}
\end{lemma}
\begin{proof}
The lemma is an easy application of Cavalieri formula. 

Let us start with (i). 
Clearly, we may suppose  $\Ln(O)<+\infty$ the inequality
being trivial otherwise.
Then, by setting $\overline{s}=(\Ln(O)/\omega_n)^{1/n}$ 
a direct integration yields\footnote{Recall that for any 
measurable set $O$ the isodiametric inequality yields 
$\overline{s}\leq\diam(O)/2$. 
}
\begin{eqnarray*}
\lefteqn{\int_{O}{|x-z|^{-\nu}}dx=
\int_0^{+\infty}\Ln(\{x\in O:\,|x-z|\leq t^{-1/\nu}\})dt
=\nu\int_0^{\diam(O)}\frac{\Ln(\{x\in O:\,|x-z|\leq s\})}{s^{1+\nu}}ds}\\&&
=\nu\left(\int_0^{\overline{s}}+
\int_{\overline{s}}^{\diam(O)}\right)\ldots ds
\leq\nu\omega_n\int_0^{\overline{s}}s^{n-\nu-1}ds
+\nu\Ln(O)\int_{\overline{s}}^{\diam(O)}s^{-\nu-1}ds\\&&
=\frac{\nu}{n-\nu}\omega_n^{\nu/n}(\Ln(O))^{1-\nu/n}
+\Ln(O)\left(-(\diam(O))^{-\nu}+
\left(\frac{\Ln(O)}{\omega_n}\right)^{-\nu/n}\right)
\leq c(n,\nu)(\Ln(O))^{1-\nu/n}.
\end{eqnarray*}

Inequality \eqref{e:Adams2} easily follows from a direct
integration. More precisely, we have 
$$
\int_{O}|x-z|^{-\nu}\dx\leq
\int_{\Rn\setminus \overline{B}_{\dist(z,O)}(z)}
|x-z|^{-\nu}\dx =
\frac{n\omega_n}{\nu-n}\left(\dist(z,O)\right)^{n-\nu}.
$$
\end{proof}

\medskip


\begin{thebibliography}{99}

\bibitem{AD} {\rm Adams R.A.}, ``Lecture Notes on $L^p$-potential theory'',
Department of Math. Univ. of Ume\aa, 1981.

\bibitem{ABCP} {\rm Alberti G., Bellettini G., Cassandro M., Presutti E.},
{\it Surface tension in Ising systems with Kac potentials}, 
J. Stat. Phys. {\bf 82} (1996), 743--796.

\bibitem{Alm-Lieb} {\rm Almgren F.J. - Lieb E.H.}, 
{\it Symmetric decreasing rearrangement is sometimes continuous}, 
J. Amer. Math. Soc. {\bf 2}  (1989),  no. 4, 683--773.

\bibitem{Amadori} {\rm Amadori A.L.}, {\it Obstacle problem for nonlinear 
integro-differential equations arising in option pricing}, 
Ric. Mat., {\bf 56} (2007), no. 1, 1–-17.

\bibitem{ANB} {\rm Ansini N. - Braides A.},
{\it Asymptotic analysis of periodically perforated nonlinear media},
J. Math. Pures Appl. (9) {\bf 81}  (2002), no. 5, 439--451.

\bibitem{Athanasopoulos} {\rm Athanasopoulos I.}, 
{\it Regularity of the solution of an 
evolution problem with inequalities on the Boundary},
Comm. Partial Differential Equations, {\bf 7} (1982), 1453--1465

\bibitem{ATT} {\rm Attouch H.}, ``Variational Convergence for Functions and 
Operators'', Applicable Mathematics Series, Pitman, Boston, 1984.

\bibitem{Bal} {\rm Balder E.J.}, {\it Lectures on Young measure theory and 
its applications in economics}, 
Rend. Istit. Mat. Univ. Trieste {\bf 31} (2000), suppl. 1, 1--69. 

\bibitem{Balzano} {\rm Balzano M.}, {\it Random relaxed Dirichlet problems}, 
Ann. Mat. Pura Appl. (4) {\bf 153} (1988), 133--174 (1989). 

\bibitem{BK1} {\rm Bass R.F. - Kassmann M.}, {\it H\"older continuity 
of harmonic functions with respect to operators of variable order}, 
Comm. Partial Differential Equations  {\bf 30} (2005),  no. 7-9, 1249--1259.

\bibitem{BK2} {\rm Bass R.F. - Kassmann M.}, {\it Harnack inequalities 
for non-local operators of variable order}, Trans. Amer. Math. Soc. 
{\bf 357} (2005),  no. 2, 837--850.

\bibitem{BLBL0} {\rm Blanc X., Le Bris C., Lions P.-L.}, 
{\it Une variant de la th\'eorie de l'homog\'en\'eisation stochastique
des op\'erateurs elliptiques}, C.R. Acad. Sci. Paris, Ser. I {\bf 343}
(2006), 717--724.

\bibitem{BLBL1} {\rm Blanc X., Le Bris C., Lions P.-L.}, 
{\it Stochastic homogenization and random lattices},  
J. Math. Pures Appl. (9)  {\bf 88}  (2007),  no. 1, 34--63.

\bibitem{BLBL2} {\rm Blanc X., Le Bris C., Lions P.-L.}, 
{\it The energy of some microscopic stochastic lattices}, 
Arch. Ration. Mech. Anal. {\bf 184} (2007),  no. 2, 303--339.

\bibitem{B3} {\rm Braides A.}, ``$\Gamma$-convergence for beginners'',
Oxford Lecture Series in Mathematics and its Applications {\bf 22}, 
Oxford University Press, Oxford, 2002.

\bibitem{BDF} {\rm Braides A. - Defranceschi A.},
``Homogenization of Multiple integrals'', Oxford University Press, 
Oxford, 1998.

\bibitem{Sigalotti} {\rm Braides, A. - Sigalotti L.} {\it Asymptotic 
analysis of periodically-perforated nonlinear media at and close to 
the critical exponent},  C. R. Math. Acad. Sci. Paris {\bf 346} (2008),  
no. 5-6, 363--367.

\bibitem{Caf-Mel1} {\rm Caffarelli L., Mellet A.}, 
{\it Random Homogenization of an Obstacle Problem}, 
 Ann. Inst. H. Poincaré Anal. Non Linéaire  {\bf 26} 
(2009), no. 2, 375--395.

\bibitem{Caf-Mel2} {\rm Caffarelli L., Mellet A.},
{\it Random Homogenization of Fractional Obstacle Problems},
Netw. Heterog. Media {\bf 3} (2008), no. 3, 523--554.

\bibitem{Caf-Sal-Silv} {\rm Caffarelli L., Salsa S., Silvestre L.},
{\it Regularity estimates for the solution and the free boundary to 
the obstacle problem for the fractional laplacian}, 
Invent. Math. {\bf 171} (2008), no. 2, 425--461.

\bibitem{Caf-Silv} {\rm Caffarelli L., Silvestre L.},
{\it An extension problem related to fractional laplacians}, 
Comm. Partial Differential Equations {\bf 32} (2007), no. 7-9, 1245--1260. 

\bibitem{Caf-Silv1} {\rm Caffarelli L., Silvestre L.},
{\it Regularity theory for fully nonlinear integro-differential equations},
Comm. Pure and Appl. Math. {\bf 62} (2009), no. 5, 597--638.

\bibitem{CD} {\rm Cioranescu D. - Donato P.}, {``An Introduction to 
Homogenization''}, 
Oxford University Press, 1999.

\bibitem{CM} {\rm Cioranescu D. - Murat F.}, {\rm Un terme \'etrange venu 
d'ailleurs, I and II}, {\it Nonlinear  Partial Differential Equations 
and Their Applications. Coll\`ege de France Seminar}, vol. II, 98-135, 
and vol. III, 154-178, Res. Notes in Math. {\bf 60} and {\bf 70}, Pitman, 
London, 1982 and 1983.

\bibitem{CMT} {\rm Conca C., Murat F., Timofte C.}, 
{\it A generalized strange term in Signorini's type problems}, 
M2AN Math. Model. Numer. Anal. {\bf 37} (2003),  no. 5, 773--805. 

\bibitem{DM2} {\rm Dal Maso G.}, {\it On the integral representation of
certain local functionals}, Ricerche Mat. {\bf 32} (1983), 85--114.

\bibitem{DM3} {\rm Dal Maso G.}, {\it Limits of minimum problems for general 
integral functionals with unilateral obstacles}, Atti Accad. Naz. Lincei 
Rend. Cl. Sci. Fis. Mat. Natur. (8) {\bf 84} fasc. 2 (1983), 55--61.

\bibitem{DM} {\rm Dal Maso G.,} {``An Introduction to 
$\Gamma\hbox{-}$convergence''}, Birkh\"auser, Boston, 1993.

\bibitem{DGDML} {\rm De Giorgi E. - Dal Maso G. - Longo P.}, 
{\it $\Gamma\hbox{-}$limiti di ostacoli}, Atti Accad. Naz. Lincei 
Rend. Cl. Sci. Fis. Mat. Natur. (8) {\bf 68} (1980), 481--487.

\bibitem{Duvaut-Lions} {\rm Duvaut G. - Lions J.L.},
"Les in\'equations en m\'echanique et en physique", Dunod, Paris, 1972

\bibitem{Fichera} {\rm Fichera G.}, {\it Problemi elastostatici con vincoli 
unilaterali: Il problema di Signorini con ambigue condizioni al contorno},
Atti Accad. Naz. Lincei Mem. Cl. Sci. Fis. Mat. Natur. Sez. I (8)  
{\bf 7}  1963/1964 91--140.

\bibitem{Foc} {\rm Focardi M.}, {\it Homogenization of random 
fractional obstacle problems via $\Gamma$-convergenqce}, 
in print on Comm. Partial Differential Equations.

\bibitem{Frehse} {\rm Frehse J.}, {\it On Signorini's problem and 
variational problems with thin obstacles},
Ann. Scuola Norm. Sup. Pisa, {\bf 4} (1977), 343--362

\bibitem{KCO} {\rm Koslowski M. - Cuiti$\mathrm{\tilde n}$o A.M. - Ortiz M.}, 
{\it A phase-field theory of dislocation dynamics, 
strain hardening and hysteresis in ductile single crystals}, 
J. Mech. Phys. Solids, {\bf 50} (2002), 2597--2635.

\bibitem{Kr} {\rm Krengel U.}, ``Ergodic Theorems'',
de Gruyter Studies in Mathematics, vol. 6, de Gruyter, 1985.

\bibitem{KrPy} {\rm Krengel U., Pyke R.}, {\it Uniform pointwise ergodic 
theorems for classes of averaging sets and multiparameter subadditive 
processes}, Stochastic Process. Appl. {\bf 26} (1987), no. 2, 289--296. 

\bibitem{MK} {\rm Marchenko V.A. - Khruslov E.Ya.}, 
{\it Boundary value problems in domains with 
fine-granulated boundaries} (in Russian), Naukova Dumka, Kiev, 1974.

\bibitem{Ng} {\rm Nguetseng G.}, {\rm Homogenization in perforated 
domains beyond the periodic setting}, J. Math. Anal. Appl. {\bf 289} 
(2004), no. 2, 608--628.

\bibitem{RT1} {\rm Rauch J. - Taylor M.},
{\it Electrostatic screening}, J. Math. Phys. {\bf 16} (1975), 284--288.

\bibitem{RT2} {\rm Rauch J. - Taylor M.}, 
{\it Potential and scattering theory on wildly perturbed domains},
J. Funct. Anal. {\bf 18} (1975), 27--59.

\bibitem{Schwab} {\rm Schwab R.}, {\it Periodic homogenization for
nonlinear integro-differential equations}, preprint 
University of Texas at Austin, 2008.

\bibitem{Sen} {\rm Senechal M.}, ``Quasicrystals and geometry'',  
Cambridge University Press, Cambridge, 1995.

\bibitem{Signorini} {\rm Signorini A.}, 
``Sopra alcune questioni di Elastostatica'',
Atti della Soc. Ital. per il Progresso della Scienze, 1963.

\bibitem{Silvestre} {\rm Silvestre L.}, PhD Thesis, 
University of Texas at Austin, 2005.

\bibitem{Tr0} {\rm Triebel H.}, ``Interpolation theory, function spaces,
differential operators'',  North-Holland mathematical library {\bf 18},
Amsterdam, 1978.

\bibitem{Tr} {\rm Triebel H.}, {\it Hardy inequality in function spaces}, 
Mathematica Bohemica {\bf 124} (1999), no. 2-3, 123--130.

\end{thebibliography}
\end{document}